\pdfoutput=1
\documentclass[reqno, noend, a4paper,oneside]{amsart}

\usepackage{amsthm,amsmath,amsfonts,mathrsfs,amssymb,todonotes,verbatim}
\usepackage{algorithm}
\usepackage{algorithmicx, algpseudocode}
\usepackage[T1]{fontenc}   
\usepackage{multirow}
\usepackage{url}
\usepackage{multirow}
\usepackage{rotating}
\usepackage{enumerate}
\usepackage{natbib}
\usepackage{epstopdf}
\usepackage{graphicx}
\usepackage{enumitem}
\usepackage{arydshln}
\usepackage[hang]{caption}
\usepackage[hyperfootnotes=false]{hyperref}
\usepackage{setspace}
\usepackage[textheight=650pt]{geometry}
\usepackage{calc}

\algdef{SE}[DOWHILE]{Do}{doWhile}{\algorithmicdo}[1]{\algorithmicwhile\ #1}%

\newtheorem{theorem}{Theorem}
\newtheorem{defn}[theorem]{Definition}
\newtheorem{prop}[theorem]{Proposition}
\newtheorem{lemma}[theorem]{Lemma}
\theoremstyle{definition}
\newtheorem{remark}[theorem]{Remark}
\newtheorem{note}[theorem]{Note}
\newtheorem{notation}[theorem]{Notation}

\newcommand{\Singular}{\textsc{Singular}}

\DeclareMathOperator{\Mon}{Mon}

\DeclareMathOperator{\id}{id}
\DeclareMathOperator{\parm}{par}
\DeclareMathOperator{\m}{\mathfrak{m}}
\DeclareMathOperator{\jet}{jet}
\DeclareMathOperator{\corank}{corank}
\DeclareMathOperator{\supp}{supp}
\DeclareMathOperator{\Jac}{Jac}

\DeclareMathOperator{\dash}{-}

\DeclareMathOperator{\NF}{NF}

\DeclareMathOperator{\N}{\mathbb{N}}
\DeclareMathOperator{\Q}{\mathbb{Q}}
\DeclareMathOperator{\R}{\mathbb{R}}
\DeclareMathOperator{\C}{\mathbb{C}}
\DeclareMathOperator{\K}{\mathbb{K}}

\DeclareMathOperator{\coeff}{coeff}


\title[A Classification Algorithm for Complex Singularities]%
{A Classification Algorithm for Complex Singularities of Corank and Modality up to Two}

\author{Janko B\"ohm}
\address{Janko B\"ohm\\
Department of Mathematics\\
University of Kaiserslautern\\
Erwin-Schr\"odinger-Str.\\
67663 Kaiserslautern\\
Germany}
\email{boehm@mathematik.uni-kl.de}

\author{Magdaleen S.\@ Marais}
\address{Magdaleen S.\@ Marais\\
University of Pretoria and African Institute for Mathematical Sciences\\
Department of Mathematics and Applied Mathematics\\
Private bag X20\\
Hatfield 0028\\
South Africa}
\email{magdaleen.marais@up.ac.za}

\author{Gerhard Pfister}
\address{Gerhard Pfister\\
Department of Mathematics\\
University of Kaiserslautern\\
Erwin-Schr\"odin\-ger-Str.\\
67663 Kaiserslautern\\
Germany}
\email{pfister@mathematik.uni-kl.de}

\thanks{This research was supported by the Staff Exchange Bursary Programme of the University of Pretoria.}


\begin{document}

\begin{abstract}

In \citep{AVG1985}, Arnold has obtained normal forms and has developed a classifier for, in particular, all isolated hypersurface singularities  over the complex numbers up to modality $2$. Building on a series of $105$ theorems, this classifier determines the type of the given singularity. However, for positive modality, this does not fix the right equivalence class of the singularity, since the values of the moduli parameters are not specified. 
In this paper, we present a simple classification algorithm for isolated hypersurface singularities of corank $\leq 2$ and modality $\leq2$.  For a singularity given by a polynomial over the rationals, the algorithm determines its right equivalence class by specifying a polynomial representative in Arnold's list of normal forms.

\end{abstract}

\maketitle

\section{Introduction}\label{Introduction}

In his classical paper on singularities \citep{A1974}, Arnold has classified all isolated hypersurface singularities over the complex numbers with modality $\leq 2$. He has given normal forms in the sense of polynomial families with moduli parameters such that every stable equivalence class of function germs contains at least one (but only finitely many) elements of these families.  We refer to such elements as normal form equations. Two germs are stably equivalent if they are right equivalent after the direct addition of a non-degenerate quadratic form. Two function germs $f,g\in\m^2\subset\C[[x_1,\ldots,x_n]]$, where $\m=\langle x_1,\ldots,x_n\rangle$, are right equivalent, written $f\sim g$, if there is a $\C$-algebra automorphism $\phi$ of $\C[[x_1,\ldots,x_n]]$ such that $\phi(f)=g$. Using the Splitting Lemma, any germ with an isolated singularity at the origin can be written, after choosing a suitable coordinate system, as the sum of two functions on disjoint sets of variables. One function that is called the non-degenerate part, is a non-degenerate quadratic form, and the other part, called the residual part is in $\m^3$. The Splitting Lemma is implemented in \Singular\ as part of the library \texttt{classify.lib} \citep{classify}. 

In \citep{AVG1985}, Arnold has made this classification explicit by describing an algorithmic classifier, which is based on a series of $105$ theorems. This approach determines the type of the singularity in the sense of its normal form. However, the values of the moduli parameters are not determined, that is, no normal form equation is given. Arnold's classifier is implemented in \texttt{classify.lib}.

Classification of complex singularities has a multitude of practical and theoretical applications. The classification of real singularities in \citep{realclassify1,realclassify2,realclassify3} is based on determining the complex type of the singularity.

In this paper, we develop a determinator for complex singularities of modality $\leq 2$ and corank $\leq 2$, which computes, for a given rational input polynomial, a normal form equation in its equivalence class. For singularities with non-degenerate Newton boundary, our determinator is based on a simple and uniform approach, which does not require a case-by-case analysis (except for some trivial final steps to read off the values of the moduli parameters according to Arnold's choice of the normal form). Two series of cases with degenerate Newton boundary are handled with more specific methods. Here, we use results of \citep{LP} to compute a normal form.  In this way, we obtain an approach which does not only determine the moduli parameters, but also allows for an elegant implementation. We have implemented our algorithm in the \Singular-library \texttt{classify2.lib} \citep{classify2}.

It is important to note that two different normal form equations do not necessarily represent two different right equivalence classes. In \citep{realclassify2} the complete structure of the equivalence classes for, in particular, complex singularities of modality $1$ and corank $2$ is determined, in the sense that all equivalences between normal form equations are described. All normal form equations in the right equivalence class of a given unimodal corank $2$ singularity can, hence, be determined by combining our classifier with the results in \citep{realclassify2}. There is not yet a similar complete description of the structure of the equivalence classes of bimodal singularities.

This paper is structured as follows: In Section \ref{Definitions and Preliminary Results}, we give the fundamental definitions and provide the prerequisites on singularities and their classification. In Section \ref{section:generalAlg}, we develop a general algorithm for the classification of complex singularities of modality $\leq 2$ and corank $\leq 2$. Essentially, the algorithm is structured into a subalgorithm for elimination below the Newton polygon, and a subalgorithm for elimination on and above the Newton polygon, which also determines the values of the moduli parameters. The algorithm for the two series of germs of modality $2$ with degenerate Newton boundary is discussed in Section \ref{section:DegenerateAlg}.

\section{Definitions and Preliminary Results}\label{Definitions and Preliminary Results}
In this section, we give some basic definitions and results, as well as some notation that will be used throughout the paper.

\begin{defn}\label{def nfequ}
Let $K\subset \C[[x_1,\ldots, x_n]]$ be a union of equivalence classes with respect to the relation ${\sim}$. A \textbf{normal form} for $K$ is given by a smooth map
\[\Phi:B\longrightarrow \C[x_1,\ldots,x_n]\subset\C[[x_1,\ldots,x_n]]\]
of a finite-dimensional $\C$-linear space $B$ into the space of polynomials for which the following three conditions hold:
\begin{itemize}[leftmargin=10mm]
\item[(1)] $\Phi(B)$ intersects all equivalence classes of $K$,
\item[(2)] the inverse image in $B$ of each equivalence class is finite,
\item[(3)] $\Phi^{-1}(\Phi(B)\setminus K)$ is contained in a proper hypersurface in $B$.
\end{itemize}
The elements of the image of $\Phi$ are called \textbf{normal form equations}.
\end{defn}

\begin{remark}
Arnold has chosen a normal form for each of the corank $2$ singularities of modality $\leq 2$. He has also associated a type to each normal form, see Table \ref{tab:normal_forms}. We denote the normal form corresponding to the type $T$ by $\NF(T)$. 
For $b\in \parm(\NF(T)):=\Phi^{-1}(K)$ with $K$ as in Definition~\ref{def nfequ}, we write $\NF(T)(b):=\Phi(b)$ for the corresponding normal form equation.
\end{remark}

\begin{table}[h]
\centering
\caption{Normal forms of singularities of modality~$\leq 2$ and corank~$\leq 2$ as given
in \citet{AVG1985}}
\label{tab:normal_forms}
\renewcommand{\thempfootnote}{\fnsymbol{mpfootnote}}
\addtocounter{mpfootnote}{1}
\newcommand{\setfnA}{\footnote{\label{fnA}%
Note that the restriction $a^2 \neq 4$ applies to the normal forms of the real
subtypes $X_9^{++}$, $X_9^{--}$, and $J_{10}^+$ as well as to the normal forms
of the complex types $X_9$ and $J_{10}$ while the restriction $a^2 \neq -4$
applies to the normal forms of the real subtypes $X_9^{+-}$, $X_9^{-+}$, and
$J_{10}^-$ if we allow complex parameters.}}
\newcommand{\reffnA}{\textsuperscript{\ref*{fnA}}}
\centering
\scalebox{0.89}{
\begin{tabular}{|c|c|c|c||c|c|c|c|}
\hline

\multicolumn{1}{|c}{}
 & & Complex         & \multirow{2}{*}{Restrictions} &\multicolumn{1}{|c}{}
 & &Complex         & \multirow{2}{*}{Restrictions} \\[-0.5ex]
\multicolumn{1}{|c}{}
 & & normal form &&\multicolumn{1}{|c}{}
 & & normal form &                              \\
\hline\hline

\multirow{7}{*}{\begin{sideways}Simple\end{sideways}}

& {$A_k$} & {$x^{k+1}$}
  &{$k\ge 1$}&\multirow{16}{*}{\begin{sideways}Bimodal\end{sideways}}&$J_{3,0}$&$x^3+bx^2y^3+y^9+cxy^7$&$4b^3+27\neq 0$
\\ \cline{2-4} \cline{6-8}

& {$D_k$} & {$x^2y+y^{k-1}$}
 &  $k\ge 4$&&$J_{3,p}$&$x^3+x^2y^3+{\bf a}y^{9+p}$&$p>0$, $a_0\neq0$ \\\cline{2-4}\cline{6-8}

& {$E_6$} & {$x^3+y^4$}
  &{-}&&$Z_{1,0}$&$x^3y+dx^2y^3+cxy^6+y^7$&$4d^3+27\neq 0$
\\ \cline{2-4}\cline{6-8}

& {$E_7$} & {$x^3+xy^3$}
  &{-}&&$Z_{1,p}$&$x^3y+x^2y^3+{\bf a}y^{7+p}$&$p>0$, $a_0\neq 0$
\\ \cline{2-4}\cline{6-8}

& {$E_8$} & {$x^3+y^5$}
  &{-}&&$W_{1,0}$&$x^4+{\bf a}x^2y^3+y^6$&$a_0^2\neq 4$\\ \cline{2-4}\cline{6-8}

& {$X_9$} & {$x^4+ax^2y^2+y^4$}
  &{$a^2\neq4$}&&$W_{1,p}$&$x^4+x^2y^3+{\bf a}y^{6+p}$&$p>0$, $a_0\neq 0$
\\ \cline{2-4}\cline{6-8}

& {$J_{10}$} & {$x^3+ax^2y^2+y^6$}
 &  $4a^3+27 \neq 0$&&$W_{1,2q-1}^\sharp$&$(x^2+y^3)^2+{\bf a}xy^{4+q}$&$q>0$, $a_0\neq 0$ \\ \cline{1-4}\cline{6-8}

\multirow{11}{*}{\begin{sideways}Unimodal\end{sideways}}

& {$J_{10+k}$} & {$x^3+x^2y^2+ay^{6+k}$}
  & {$a \neq 0,\; k > 0$} &&$W_{1,2q}^\sharp$&$(x^2+y^3)^2+{\bf a}x^2y^{3+q}$&$q>0$, $a_0\neq 0$ \\ \cline{2-4}\cline{6-8}

& {$X_{9+k}$} &{$x^4+x^2y^2+ay^{4+k}$}
  
      & {$a \neq 0,\; k > 0$}&&$E_{18}$&$x^3+y^{10}+{\bf a}xy^7$&-  \\
 \cline{2-4}\cline{6-8}

&{$Y_{r,s}$} &{$x^r+ax^2y^2+y^s$}
 &{$a \neq 0,\; r,s > 4$}&&$E_{19}$&$x^3+xy^{7}+{\bf a}y^{11}$&- \\ 
\cline{2-4}\cline{6-8}

& $E_{12}$ & $x^3+y^7+axy^5$ & -&&$E_{20}$&$x^3+y^{11}+{\bf a}xy^8$&- \\ \cline{2-4}\cline{6-8}

& $E_{13}$ & $x^3+xy^5+ay^8$ &  - &&$Z_{17}$&$x^3y+y^8+{\bf a}xy^6$&-\\ \cline{2-4}\cline{6-8}

&$E_{14}$ & {$x^3+y^8+axy^6$}
   &{-}&&$Z_{18}$&$x^3y+xy^6+{\bf a}y^9$&-\\ \cline{2-4}\cline{6-8}

& $Z_{11}$ & $x^3y+y^5+axy^4$ & -&&$Z_{19}$&$x^3y+y^9+{\bf a}xy^7$&- \\ \cline{2-4}\cline{6-8}

& $Z_{12}$ & $x^3y+xy^4+ax^2y^3$& -&&$W_{17}$&$x^4+xy^5+{\bf a}y^7$&- \\ \cline{2-4}\cline{6-8}

& {$Z_{13}$} & {$x^3y+y^6+axy^5$}
  & {-} &&$W_{18}$&$x^4+y^7+{\bf a}x^2y^4$&-\\ \cline{2-4}\cline{6-8}

&{$W_{12}$} &{$x^4+y^5+ax^2y^3$}
  &{-} &&&&\\ \cline{2-4}

&{$W_{13}$} & {$x^4+xy^4+ay^6$}
  &{-} &&&where $\mathbf{a}=a_0+a_1 y$&
    \\ \hline

\end{tabular}
}
\end{table}

In the following, we give a short account on weighted jets, filtrations, and Newton polygons. See
\citep{A1974} and \citep{PdJ2000} for more details.

\begin{defn}
Let $w=(c_1,\ldots,c_n)\in\N^n$ be a weight on the variables $(x_1,\ldots,x_n)$. The $w$-weighted degree on $\Mon(x_1,\ldots,x_n)$ is given by $w\dash\deg(\prod_{i=1}^nx_i^{s_i}):=\sum_{i=1}^n c_i s_i$. If the weight of all variables is equal to $1$, we refer to the weighted degree of a monomial $m$ as the standard degree of $m$ and write $\deg(m)$ for $w\dash\deg(m)$. We use the same notation for terms of polynomials.

We call a polynomial $f\in \C[x_1,\ldots, x_n]$ \textbf{quasihomogeneous} or weighted homogeneous of degree $d$ with respect to the weight $w$
if $w\dash\deg(t)=d$ for any term $t$ of $f$.
\end{defn}

\begin{defn}\label{def:piecewiseWeight}
Let $w=(w_1,\ldots,w_s)\in(\N^n)^s$ be a finite family of weights on the variables $(x_1,\ldots,x_n)$. For any monomial (or term) $m\in \C[x_1,\ldots,x_n]$,  we define the \textbf{piecewise weight} with respect to $w$ as
\begin{eqnarray*}
w\dash\deg(m)&:=&\min_{i=1,\ldots,s}w_i\dash\deg(m).\\
\end{eqnarray*}
A polynomial $f$ is called  \textbf{piecewise homogeneous} of degree $d$ with respect to $w$ if $w\dash\deg(t)=d$ for any term $t$ of $f$. 
\end{defn}
\pagebreak[3]

\begin{defn}
Let $w$ be a (piecewise) weight on $\Mon(x_1,\ldots,x_n)$.
\begin{enumerate}[leftmargin=10mm]
\item
Let $f = \sum_{i = 0}^{\infty} f_{i}$ be the decomposition of
$f \in \C[[x_1,\ldots,x_n]]$ into weighted homogeneous summands $f_{i}$ of
$w$-degree $i$. The \textbf{weighted $j$-jet} of $f$ with respect to $w$ is
\[
w \dash \jet(f, j) := \sum_{i = 0}^j f_{i} \,.
\]
The sum of terms of  $f$ of lowest $w$-degree is the \textbf{principal part} of $f$ with respect to $w$.

\item A power series in $\C[[x_1,\ldots,x_n]]$ has \textbf{filtration} $d \in \N$ with respect to $w$ if all its
monomials are of $w$-weighted degree $d$ or higher. The power series of filtration
$d$ form a sub-vector space
\[
E_d^w \subset \C[[x_1,\ldots,x_n]] \,.
\]

\item A power series $f\in\C[[x_1,\ldots,x_n]]$ is \textbf{weighted $k$-determined} with respect to the weight $w$ if
\[ f\sim w\dash\jet(f,k)+g\qquad\text{for all } g\in E_{k+1}^w.\]
We define the \textbf{weighted determinacy} of $f$ as the minimum number $k$ such that $f$ is $k$-determined.
\end{enumerate}
\end{defn}

\begin{defn}
Let $w\in \mathbb{N}^n$ be a single weight.  A power series $f\in \C[[x_1,\ldots,x_n]]$ is called \textbf{semi-quasihomogeneous} with respect to $w$ if its principal part with respect to $w$ is non-degenerate, that is, has finite Milnor number.\footnote{We say that $f$ is (semi-)quasihomogeneous if there exists a weight $w$ such that $f$ is (semi-)quasihomogeneous with respect to $w$.} The principal part is then called the \textbf{quasihomogeneous part} of $f$.

\end{defn}

\begin{notation}
\begin{enumerate}[leftmargin=10mm]
\item If the weight
of each variable is $1$, we write $E_d$ and $\jet(f,j)$ instead of  $E_d^w$ and $w\dash\jet(f,j)$, respectively. 
\item If for a given type $T$, $w\dash\jet(\NF(T)(b),j)$ is independent of $b\in\parm(\NF(T))$, we denote it by $w\dash\jet(T,j)$. 
\end{enumerate}
\end{notation}

There are similar concepts of jets and filtrations for coordinate transformations:

\begin{defn}\label{phi}
Let $\phi$ be a $\C$-algebra automorphism of $\C[[x_1,\ldots,x_n]]$ and let
$w$ be a weight on $\Mon(x_1,\ldots,x_n)$.

\begin{enumerate}[leftmargin=10mm]
\item
For $j > 0$ we define \emph{$w\dash\jet(\phi,j):=\phi_j^w$} as the automorphism given by
\[
\phi_j^w(x_i) := w\dash\jet(\phi(x_i),w\dash\deg(x_i)+j) \quad
\text{for all }i = 1,\ldots,n \,.
\]
If the weight of each variable is equal to $1$, that is, $w = (1, \ldots, 1)$, we
write $\phi_j$ for $\phi_j^w$.

\item\label{enum:filtration}
$\phi$ has filtration $d$ if, for all $\lambda \in \N$,
\[
(\phi-\id)E_\lambda^w \subset E_{\lambda+d}^w \,.
\]
\end{enumerate}
\end{defn}

\begin{remark}
Note that $\phi_0(x_i) = \jet(\phi(x_i), 1)$ for all $i = 1, \ldots, n$.
Furthermore note that $\phi_0^w$ has filtration $\le 0$, and
that, for $j > 0$, $\phi_j^w$ has filtration $j$ if $\phi_{j-1}^w = \id$.
\end{remark}

The following definition gives an infinitesimal analogue of the above definition.
\begin{defn}
A formal vector field ${\bf v}=\sum_i v_i\frac{\partial}{\partial x_i}$ has filtration $d$ with respect to a weight $w$, if the directional derivative of ${\bf v}$ raises the filtration by not less than $d$, that is,
\[\text{for all }g\in E^w_{\delta}, \quad L_{\bf v}(g):=\sum_i v_i\frac{\partial g}{\partial x_i}\in E^w_{\delta+d}.\]
\end{defn}

In a similar way as \cite[Proposition 8]{realclassify1}, one can prove:

\begin{prop}\label{wfact}
Let $f,g \in \C[[x_1,\ldots,x_n]]$ be two power series with $f \sim g$. Let $w\in\N^n$ and suppose that the maximal weighted filtration of $f$ with respect to $w$ is $k$. Furthermore, let $\phi$ be a $\C$-algebra automorphism of
$\C[[x_1,\ldots,x_n]]$ such that $\phi(f)=g$.
If $\jet(f,k)$ factorizes as
\[
w\dash\jet(f,k) = f_1^{s_1} \cdots f_t^{s_t}
\]
in $\C[x_1,\ldots,x_n]$, then $w\dash\jet(g,k)$ factorizes as
\[
w\dash\jet(g,k) = \phi_0^w(f_1)^{s_1} \cdots \phi_0^w(f_t)^{s_t} \,.
\]
\end{prop}

 \begin{defn}
 Let $f=\sum_{i,j}a_{i,j}x^{i}y^{j}\in\C[[x,y]]$, let $T$ be a corank $2$ singularity type. We call
 \begin{eqnarray*}
 \supp(f)&:=&\{x^{i}y^{j}\ |\ a_{i,j}\neq 0\}\\
 \supp(T)&:=&\supp(\NF(T)(b))
 \end{eqnarray*}
where $b\in\parm(\NF(T))$ is generic, the \textbf{support} of $f$ and of $T$, respectively. Let
\begin{eqnarray*}
\Gamma_+(f)&:=&\displaystyle{\bigcup_{x^{i} y^{j}\in\supp(f)}}((i,j)+\R^2_+)\\
\Gamma_+(T)&:=&\displaystyle{\bigcup_{x^{i} y^{j}\in\supp(T)}}((i,j)+\R^2_+)
\end{eqnarray*}
and let $\Gamma(f)$ and $\Gamma(T)$ be the boundaries in $\R^2$ of the convex hulls of $\Gamma_+(f)$ and $\Gamma_+(T)$, respectively. Then:
\begin{enumerate}[leftmargin=10mm]
\item $\Gamma(f)$ and $\Gamma(T)$ are called the \textbf{Newton polygons} of $f$ and $T$, respectively.
\item The compact segments of $\Gamma(f)$ or $\Gamma(T)$ are called \textbf{faces}. If $\Delta$ is a face, then the set of monomials of $f$ lying on $\Delta$ is denoted by $\supp(f,\Delta)$ and the sum of the terms lying on $\Delta$ by $\jet(f,\Delta)$. 
Moreover, we write $\supp(\Delta)$ for the set of monomials corresponding to the lattice points of $\Delta$, and set $\supp(T,\Delta):=\supp(T)\cap\supp(\Delta)$.  
We use the same notation for a set of faces, considering the monomials lying on the union of the faces.
\item Any face $\Delta$ induces a weight $w(\Delta)$ on $\Mon(x,y)$ in the following way: If $\Delta$ has slope $-\frac{w_x}{w_y}$, in lowest terms, and $w_x,w_y>0$, we set $w(\Delta)\dash\deg(x)=w_x$ and $w(\Delta)\dash\deg(y)=w_y$. 
 \item If $w_1,\ldots,w_s$ are the weights associated to the faces of $\Gamma (f)$, respectively $\Gamma(T)$, ordered by increasing slope, there are unique minimal integers $\lambda_1,\ldots,\lambda_s\geq 1$ such that the piecewise weight associated to $(\lambda_1 w_1,\ldots,\lambda_s w_s)$ by Definition \ref{def:piecewiseWeight} is constant on $\Gamma(f)$, respectively $\Gamma(T)$.  We denote this piecewise weight by $w(f)$, respectively $w(T)$, and the corresponding constant by $d(f)$, respectively $d(T)$.
\item Let $\Delta_i$ and $\Delta_j$ be faces with weights $w_1$ and $w_2$, respectively, and let $w$ be the piecewise weight defined by $w_1$ and $w_2$. Let $d$ be the $w$-degree of the monomials on $\Delta_1$ and $\Delta_2$. Then $\operatorname{span}(\Delta_1,\Delta_2)$ is the Newton polygon associated to the sum of all monomials of $w$-degree $d$.
\item A monomial $m$ lies strictly underneath, on or above $\Gamma(f)$, if the $w(f)$-degree of $m$ is less than, equal to or greater than $d(f)$, respectively. We use this notation also with respect to $\Gamma(T)$, $w(T)$, and $d(T)$.
 \end{enumerate}
 \end{defn}

 \begin{notation}
Given  $f\in\mathbb{C}[[x_1,\ldots, x_n]]$ and $m\in \Mon(x_1,\ldots, x_n)$, we write $\coeff(f,m)$ for the coefficient of $m$ in $f$. 
\end{notation}

 \begin{defn}
 The \textbf{Jacobian ideal} $\Jac(f)\subset \C[[x_1,\ldots x_n]]$ of $f$ is generated by the partial derivatives of $f\in\mathbb{C}[[x_1,\ldots x_n]]$. The \textbf{local algebra} of $f$ is the residue class ring of the Jacobian ideal of $f$.
 \end{defn}

\begin{defn}
Suppose $f$ is a non-degenerate germ, $e_1,\ldots,e_\mu$ are monomials representing a basis of the local algebra of $f$, and $e_1,\ldots,e_s$ are the monomials in this basis above or on $\Gamma(f)$. We then call $e_1,\ldots,e_s$ a \textbf{system} of the local algebra of $f$.
\end{defn}

\begin{lemma}[\citet{A1974}, Corollary 3.3]
Let $f$ be a semi-quasihomogeneous function with quasihomogeneous part $f_0$, and let $e_1,\ldots,e_\mu$ be monomials representing a basis of the local algebra of $f_0$. Then $e_1,\ldots,e_\mu$ also represent a basis of the local algebra of $f$.
\end{lemma}

 \begin{theorem}[\citet{A1974}, Theorem 7.2]\label{thm:quasiNF}
 Let $f$ be a semi-quasihomogeneous function with quasihomogeneous part $f_0$ 
 and let $e_1,\ldots,e_s$ be a system of the local algebra of $f$.
 Then $f$ is equivalent to a function of the form $f_0+\sum_{k=1}^{s} c_ke_k$ with constants $c_k$. 
 \end{theorem}

 In \cite{A1974}, the following approach is used to extend the above results to a larger class of singularities of corank~$2$. 

\begin{defn}
A piecewise homogeneous function $f_0$ of degree $d$ satisfies \textbf{Condition A}, if for every function $g$ of filtration $d+\delta>d$ in the ideal spanned by the derivatives of $f_0$, there is a decomposition
\[g=\sum_i\frac{\partial f_0}{\partial x_i}v_i+g',\]
where the vector field $v$ has filtration $\delta$ and $g'$ has filtration bigger than $d+\delta$.
\end{defn}
 
Note that quasihomogeneous functions satisfy Condition A.

\begin{theorem}\label{principalpart}Suppose that the principal part $f_0$ of the piecewise homogeneous function $f$ has finite Milnor number and satisfies Condition A. Let $e_1,\ldots,e_s$ be a system of the local algebra of $f_0$.
Then $f$ is equivalent to a function of the form $f_0+\sum_k c_ke_k$ with constants $c_k$. \end{theorem}

Following Arnold's proof of Theorem \ref{thm:quasiNF}, Theorem \ref{principalpart} can be proven by iteratively applying the following lemma.

\begin{lemma}\label{lemma:quasiNF}
Let $f_0\in\C[[x_1,\ldots,x_n]]$ be a piecewise homogeneous function of weighted $w$-degree $d_w$ that satisfies Condition A, and let $e_1,\ldots,e_r$ be the monomials of a given $w$-degree $d'>d_w$ in a system of the local algebra of $f_0$. Then, for every series of the form $f_0+f_1$, where the filtration of $f_1$ is greater than $d_w$, we have
\[f_0+f_1\sim f_0+f'_1,\]
where the terms in $f'_1$ of degree less than $d'$ are the same as in $f_1$, and the part of degree $d'$ can be written as $c_1e_1+\cdots+c_re_r$ with $c_i\in\C$.
\end{lemma}

\begin{proof}
Let $g({\bf x})$ denote the sum of the terms of degree $d'$ in $f_1$. There exists a decomposition of $g$ of the form
\[g({\bf x})=\sum_i\frac{\partial f_0}{\partial x_i}v_i({\bf x})+c_1e_1+\cdots +c_re_r,\qquad v_i\in\C[[x_1,\ldots,x_n]],\]
since $e_1,\ldots,e_r$ represent a monomial vector space basis of  the local algebra of $f_0$ in degree $d'$. Let $d(x_i)$ be the $w$-degree of $x_i$, and let $v'_i:=w\dash\jet(v_i,d(x_i))$. Then
\[g({\bf x})=\sum_i\frac{\partial f_0}{\partial x_i}v'_i({\bf x})+c_1e_1+\cdots +c_re_r- g'({\bf x}),\]
where $g'({\bf x})$ has filtration greater than $d'$. 
 Applying the transformation defined by
\[x_i\mapsto x_i-v'_i({\bf x})\]
to $f=f_0+f_1$, we transform $f$ to
\[f_0({\bf x})+\left( f_1({\bf x})+(c_1e_1({\bf x})+\cdots+c_re_r({\bf x})-g({\bf x})\right)+R({\bf x}),\]
where the filtration of $R$ is greater than $d'$.
\end{proof}

 \begin{remark} 
A system of the local algebra is in general not unique. For his lists of normal forms  of hypersurface singularities
, Arnold has chosen  in each case (in particular) a specific system of the local algebra. In the rest of the paper, we call these systems the \textbf{Arnold systems}.
 \end{remark}

\begin{defn}[\cite{K1976}]
We say that $f\in \C[[x,y]]$ has \textbf{non-degenerate Newton boundary} if for every face $\Delta$ of $\Gamma(f)$ the saturation of $\jet(f,\Delta)$ has finite Milnor number.\footnote{We say that the singularity defined by $f$ has non-degenerate Newton boundary  if there exists a germ $\tilde{f}\in \C[[x,y]]$ with $f\sim \tilde{f}$ which has non-degenerate Newton boundary. We use the analogous terminology also for semi-quasihomogeneous.
}
\end{defn}

\begin{remark}\label{rmk elimination}
\begin{enumerate}
\item Note that if $f$ has non-degenerate Newton boundary and finite Milnor number, then the principal part of $f$ with respect $w(f)$ has finite Milnor number.
\item Also note that for Arnold's normal forms $\NF(T)$ of corank and modality $\leq 2$ the principal part with respect to $w(T)$ satisfies Condition A. 
\item Suppose that $f$ is a function of corank $2$ with non-degenerate Newton boundary such that, for one of Arnold's normal forms $\NF(T)$ of modality $\leq 2$, the support of the principal part $f_0$ of $f$ with respect to $w(T)$ coincides with that of the principal part of $\NF(T)$. Then 
a system of the local algebra of $f_0$ is also a system of the local algebra of $f$.
\end{enumerate}
\end{remark}

 \begin{remark}
It follows from Lemma \ref{lemma:quasiNF} 
that all hypersurface singularities of corank $\leq 2$ and modality $\leq 2$ with non-degenerate Newton boundary are finitely weighted determined. Moreover, we explicitely obtain the weighted determinacy for each such singularity.
\end{remark}

\section{A Classification Algorithm for Corank 2 Complex Simple, Unimodal and Bimodal Singularities 
}
\label{section:generalAlg}

We now describe an algorithm to determine an Arnold normal form equation for a given input polynomial $f\in\m^3$, $f\in\Q[x,y]$ of modality $\le 2$.  In this section, we limit our discussion on functions with a normal form with non-degenerate Newton boundary. In the case of normal forms with degenerate Newton boundary, our algorithm will resort to special algorithms described in Section \ref{section:DegenerateAlg}.  
Figures \ref{fig infinite} to 
\ref{fig W} illustrate the modality $2$ types with non-degenerate Newton boundary. The figures show in the gray shaded area all monomials which can possibly occur in a polynomial $f$ of the given type $T$
. The faces of the Newton polygon $\Gamma(T)$ are shown in blue. The dots with a thick black circle indicate the moduli monomials in the Arnold system. Red dots indicate monomials which are not in $\Jac(f)$. Monomials occuring in any normal form equation with non-zero coefficients are shown as blue dots.

The structure of our algorithm consists out of two basic steps, see Algorithm \ref{alg:Classification}. We first determine the complex type of $f$ by removing all the monomials underneath $\Gamma(T)$, in the semi-quasihomogeneous cases, and all the monomials underneath and on $\Gamma(T)$ which are not in $\NF(T)$, in the other cases (Algorithm \ref{alg:clas}). After that, we determine a normal form equation of $f$ (using Algorithm~\ref{alg:parameter} in the non-simple cases). More generally, we will formulate the algorithm in a way, that it is applicable to any $f\in\m^2$, and will recognize if $f$ is of modality $>2$, returning an error in this case.

\begin{algorithm}[ht]
\caption{Algorithm to classify singularities of modality $\leq 2$ corank $\leq 2$}%
\label{alg:Classification}
\begin{algorithmic}[1]

\Require{A polynomial germ $f\in\m^2$ over the rationals.}
\Ensure{
$\NF(f)$ as well as the values of all moduli parameters occuring in a normal form equations that is equivalent to $f$, if $f$ is of modality $\leq 2$, corank $\leq 2$; \texttt{false} otherwise.
}

\State Apply Algorithm \ref{alg:clas} to $f$.
\If{$T$ as returned by Algorithm \ref{alg:clas} is a simple type}
    \Return{$(\NF(T),())$}
\EndIf
\State Apply Algorithm \ref{alg:parameter} to the output of Algorithm \ref{alg:clas} and return the result.

\end{algorithmic}
\end{algorithm}

We first discuss Algorithm \ref{alg:clas}. If $f$ is of corank $\leq 1$, then $f$ is of type $A_k$, where $k=\mu(f)$. 
Suppose now that $f$ is of corank $2$. Determining $T$ in the process, we remove all monomials below $\Gamma(T)$ if $\Gamma(T)$ has only one face, and all monomials on or below $\Gamma(T)$ which are not in $\NF(T)$, if $\Gamma(T)$ has two faces.  Let $d$ be the maximal filtration of $f$.  If $f$ is of type $X_9$, nothing has to be done. Note that $f$ is of type $X_9$ if and only if the $d$-jet of $f$ has $4$ different roots over the complex numbers. If $f$ is not of type $X_9$, then Algorithm \ref{alg:linearTransformation} will transform $f$ such that $\supp(T,d)=\supp(\jet(f,d))$. Using \cite[Proposition 8]{realclassify1}, we find the corresponding linear transformation by factorizing $\jet(f,d)$.

At this stage we know that $\supp(\jet(f,d))\subset\supp(\NF(T))$. We store the monomials of the $d$-jet of $f$ in $S_0=\supp(\jet(f,d))$. The remainder of Algorithm \ref{alg:clas} will proceed in an iterative way, changing $f$ and $S_0$ in the process:  In each step of the iteration, we can have one of the following two possibilities for $\Gamma(f)$:
\begin{enumerate}[leftmargin=*]
\item\label{two face} Note that monomials of the form $x^{n_1}y$ or $xy^{n_2}$ cannot be intersection points of (finite) faces of $\Gamma(T)$.
If any of the monomials $m_0\in S_0$ which is not of this form lies on two faces of $\Gamma(f)$, it is clear that $\Gamma(T)$ has at least two faces with corner point $m_0$. The algorithm will then stay in this case. Let  $\Delta_i$ and $\Delta_j$ be the two different faces of $\Gamma(f)$ on which $m_0$ lies. The corner point in all modality $1$ and $2$ cases with a Newton polygon with two faces is either $x^2y^2$ or $x^2y^3$. It follows that if $m_0\neq x^2y^t$, $t=2$ or $t=3$, then $f$ is not of modality $\leq 2$. Otherwise, using the shape of $\Gamma_0:=\operatorname{span}(\Delta_i,\Delta_j)$ and the fact that $m_0=x^2y^t$ is a corner point of $\Gamma_0$, all monomials in $f$ on $\Gamma_0$ of the form $xy^{n}$ or $x^{n}y^{t-1}$ can be removed iteratively, by increasing degree, each time replacing the corresponding terms of the given degree by higher $w(f)$-degree terms using Algorithm \ref{alg:Transformation}. After each iteration, $f$, $\Delta_i$, $\Delta_j$ and $\Gamma_0$ are recalculated. In each iteration, there will either be no terms of the considered form on $\Gamma_0$, in which case the process stops, or the number of equivalence classes in the local algebra of $f$ represented by powers of $x$ or $y$ underneath $\Gamma_0$ strictly increases, except possibly in the last two steps of the process (where monomials on the final Newton polygon may be removed). Note that, if $x^{m_1}$ and $y^{m_2}$ are largest powers of $x$ and $y$ underneath $\Gamma_0$, then $1,x,\ldots,x^{m_1-1},y,\ldots,y^{m_2-1}$ represent different equivalence classes. Since $\mu(f)$ is finite, the process must stop after finitely many iterations. No further monomials on $\Gamma_0$ can be removed without creating terms underneath $\Gamma_0$.
Hence, in all cases in consideration, this algorithm will produce the Newton polygon of the normal form. In fact, if $\supp(f,\Gamma_0)$
does not coincide with $\supp(T,\Gamma_0)$
for some type $T$ of modality $\leq 2$, then the modality of $f$ is bigger than $2$.
Otherwise, $f$ is a germ of the corresponding type $T$, and all monomials in $f$ underneath or on $\Gamma(T)$ not in $\NF(T)$ are removed.\smallskip

\item Suppose no monomials in $S_0$, except monomials of the form $x^{n_1}y$ or $xy^{n_2}$, lie on two faces of $\Gamma(f)$. Then $f$ is not of type $X_{9+k}$ or $Y_{r,s}$, since these cases will be recognized to have two faces in the first iteration of the above step. All the monomials in $S_0$ lie on only one face of $\Gamma(f)$. Let $\Delta$ be this face. If $f_1:=\jet(f,\Delta)$ is non-degenerate, then $f$ is a semi-quasihomogeneous germ. Since $w\dash\jet(\phi_0^w(f),d(f))=\phi_0^w(f_1)$ for any automorphism $\phi$ of filtration $\ge 0$ with respect to the weight $w$ associated to $\Delta$, $\operatorname{span}(\Delta)$ is an invariant of the type of $f$. The corresponding type $T$ can, hence, be identified. The case $X_9$ will already be recognized as a semi-quasihomogeneous function in the first iteration, and $f$ will be returned by the algorithm without any change. 
In all other cases, the weight $w$ associated with $\Delta$ will be such that $w\dash\deg(x)>w\dash\deg(y)$. If $f_1$ is degenerate, then either $f$ has monomials underneath $\Gamma(T)$, or $\Gamma(T)$ is degenerate. For all semi-quasihomogeneous cases of modality $\leq 2$, except $X_9$, $\jet(T,d)$ is divisible by a power of $x$, and $x$ has the highest multiplicity among all prime factors. Any weighted jet of $\NF(T)$ with respect to a face lying below $\Gamma(T)$ and intersecting $\Gamma(T)$ in $\jet(T,d)$ has the same property. 
 Suppose $\Delta$ is such a face. Then $\supp(T,\Delta)=\{x^ny^m\}$ with $n>m$. Taking into account that the weighted degree of $x$ is greater than the weighted degree of $y$, it follows that $f_1=g_1^ny^m$ with $\deg_x(g_1)=1$.
The right equivalence $g_1\mapsto x$, $y\mapsto y$ transforms $f$ such that $\supp(f,\Delta)=\supp(T,\Delta)$.
If the normal form of $f$ has a non-degenerate Newton boundary, but is not semi-quasihomogeneous, we can proceed in the same way: Suppose $\Delta$ lies underneath or on the face of biggest slope of $\Gamma(T)$. If $g_1$ is the factor of highest multiplicity of $f_1$ with $\deg_x(g_1)=1$, then the right equivalence $g_1\mapsto x$, $y\mapsto y$ transforms $f$ such that $\supp(f,\Delta)=\supp(T,\Delta)$. We then update $S_0:=\supp(f,\Delta)$ and pass to the next iteration. If $f_1$ does not have any $x$-linear factor, then the normal form of $f$ has a degenerate Newton boundary. In this case, we resort to the algorithms described in Section \ref{section:DegenerateAlg}. Since $\mu(f)$ is finite, the same argument as in (\ref{two face}) shows that the iteration terminates after finitely many steps.
\end{enumerate}

\begin{algorithm}[p]
\caption{Determine the complex type of a corank $\leq 2$ singularity of modality $\le 2$ with non-degenerate Newton boundary.}%
\label{alg:clas}
\begin{spacing}{1.05}
\begin{algorithmic}[1]

\Require{A polynomial germ $f\in\m^2$ over the rationals.}

\Ensure{If $f$ is of modality $\leq 2$ and corank $\leq 2$, then the complex singularity type $T$ of $f$, and a polynomial $g$ right equivalent to $f$ such that the set of faces of $\Gamma(T)$ equals the set of faces of $\Gamma(w(T)\dash\jet(g,d(T))$; \texttt{false} otherwise.
}
\State $f:=$ residual part given by the splitting lemma applied to $f$, as implemented in \texttt{classify.lib}.
\If{$\corank(f)\leq 1$}\label{Begin Case A}
\Return $(f, A_{\mu(f)})$ 
\EndIf\label{End Case A}
\If{$\corank(f)>2$}
\Return \texttt{false}
\EndIf
\If{$f\in E_5$}
\Return \texttt{false} (modality $>2$)
\EndIf
\State $f$:= output of Algorithm \ref{alg:linearTransformation} applied to $f\in \Q[x,y]$ 
\State $S_0:=\supp(\jet(f,d))$, where $d:=$ maximal filtration of $f$  w.r.t.~the standard grading.
\While{\textbf{true}}
\State Let $\Delta_1,\ldots,\Delta_n$ be the faces of $\Gamma(f)$ ordered by increasing slope.
\If{exist $i\neq j$ and $n_1, n_2 >1$ with $m_0:=x^{n_1}y^{n_2}\in\supp(\Delta_i)\cap\supp(\Delta_j)\subset S_0$}
\If{$m_0\neq x^2y^2$ and $m_0\neq x^2y^3$}
\Return \texttt{false} (modality $>2$)
\EndIf
\State $\Gamma_0:=\operatorname{span}(\Delta_i,\Delta_j)$
\State $f_1:=\jet(f,\Gamma_0)$
\While{exist a term of the form $t=c\cdot x^{n_1-1}y^{r}$ or $t=c\cdot x^{r}y^{n_2-1}$ in $f_1$}
\State $f$:= output of Algorithm \ref{alg:Transformation} with input $f$, $f_1$, $t$, and weights $w(\Delta_i), w(\Delta_j)$
\State Let $\Delta_1,\ldots,\Delta_n$ be the faces of $\Gamma(f)$.
\State $\Gamma_0:=\operatorname{span}(\Delta_i,\Delta_j)$, where $i$ and $j$ are such that  $m_0\in\supp(\Delta_i)\cap\supp(\Delta_j)$
\State $f_1:=\jet(f,\Gamma_0)$
\EndWhile
\If{exists modal $1$ or $2$ type $T$ with 
  $\supp(T,\Gamma_0)=\supp(f_1)$}
  \Return{($f$, $T$)}
\Else
  \Return \texttt{false} (modality $>2$)
\EndIf
\Else
\State Let $\Delta$ be the face of $\Gamma(f)$ of smallest slope such that $S_0\subset\supp(\Delta)$.
\State $f_1:=\jet(f,\Delta)$
\If{$\mu(f_1)=\infty$ }
  \State Let $g_1$ be the factor of $f_1$ with highest multiplicity.
  \If{$\deg_x(g_1)=1$}
    \State Replace $f$ by $g_1\mapsto x$, $y\mapsto y$ applied to $f$.
    \State $S_0 :=\supp(\jet(f,\Delta))$
  \Else
    \If{ $\supp(f,\Delta)=\supp((y^2-x^3)^2)$}
      \Return{($f$, $W_{1,\mu(f)-15}^\sharp$)}
     \Else
      \Return{\texttt{false} (modality $>2$)}
    \EndIf
\EndIf
  \Else
\If{exists modal $1$ or $2$ type $T$ with
  $\Gamma(T)=\Gamma(f_1)$}
  \Return{($f$, $T$)}
\Else
  \Return \texttt{false} (modality $>2$)
\EndIf

\EndIf
\EndIf
\EndWhile
\end{algorithmic}
\end{spacing}

\end{algorithm}

\begin{algorithm}[hp]
\caption{Reverse linear jet.
}%
\label{alg:linearTransformation}
\begin{spacing}{1.05}
\begin{algorithmic}[1]

\Require{A polynomial $f\in\m^3\subset\Q[x,y]$ with $\jet(f,4)\neq 0$.
}
\Ensure{$g\in\m^3\subset\K[x,y]$, where $\K$ is an algebraic extension field of $\Q$, such that $g\sim f$, and, in case $f$ is of type $T\neq X_9$ of modality $\le 2$, then $\supp(\jet(g,d))=\supp(T,d)$ where $d$ is the maximal filtration of $f$ w.r.t.~the standard grading.}
\State Factorize $\jet(f,d)=c g_1^\alpha g_2^\beta g_3^\gamma g_4^\delta$ over $\C$, where $0\neq c\in\Q$, $g_1$, $g_2$, $g_3$ and $g_4$ are monic in $x$ and pairwise coprime, and $4\ge \alpha\ge \beta\ge \gamma\ge \delta\ge0$.\label{Begin Max Filtration}
\If{$\beta,\gamma,\delta=0$}
  \If{$g_1\neq c'y$, $c'\in\Q$}
  \State Replace $f$ with $g_1\mapsto x$, $y\mapsto y$ applied to $f$. 
  \Else
  \State Replace $f$ with $x\mapsto y$, $y\mapsto x$ applied to $f$.
  \EndIf
\EndIf
\If{$\gamma,\delta=0$}
	\State Replace $f$ with  $g_1\mapsto x$ and $g_2\mapsto y$ applied to $f$.
\EndIf
\If{$\alpha=2$ and  $\beta,\gamma=1$ and $\delta=0$}
\If{$g_1\neq c'y$, $c'\in\Q$}
\State Replace $f$ with $g_1\mapsto x$, $y\mapsto y$ applied to $f$.
\Else
\State Replace $f$ with $x\mapsto y$, $y\mapsto x$ applied to $f$.
\EndIf
\State Write $f=a_0x^4+a_1x^3y+a_2x^2y^2+R$, $a_0,a_1\in\Q$, $a_2\in\Q^{\times}$ and $R\in E_5$.
\State Replace $f$ with $y\mapsto y-\frac{a_1}{2a_2}x$, $x\mapsto x$ applied to $f$.
\EndIf 
\Return $f$
\end{algorithmic}
\end{spacing}
\end{algorithm}

\begin{algorithm}[hp]
\caption{Remove term via partials.
}%
\label{alg:Transformation}
\begin{spacing}{1.05}
\begin{algorithmic}[1]
\Require{$f,f_0\in\K [x,y]$ over a field $\K$, with $t$ a term of $f$, and weights $u_1,u_2\in\mathbb Z^2$.}
\Ensure{$g\in\K [x,y]$ such that $f\sim g$.
If called with input as in Algorithms \ref{alg:clas} or \ref{alg:parameter}, then $f=g+t+$ terms of higher $(u_1,u_2)\dash$degree than $t$. 
}
\State $m_x:=$ the sum of the terms of $\frac{\partial f_0}{\partial x}$ of lowest $u_2$-degree 
\State $m_{x,y}:=$ the term of $m_x $ of lowest $u_1$-degree 
\State $m_y:=$ the sum of the terms of $\frac{\partial f_0}{\partial y}$ of lowest $u_1$-degree
\State $m_{y,x}:=$ the term of $ m_y$ of lowest $u_2$-degree 
\If{$m_{x,y}|t$}\label{line:OneFace}
\State \vspace{-4mm}\label{line:OneFaceEnd}\begin{eqnarray*}
\alpha:\K[x,y]&\rightarrow&\K[x,y]\\
       x&\mapsto& x-t/m_{x,y}\\
       y&\mapsto& y
\end{eqnarray*}
\Return $\alpha(f)$
\EndIf
\If{$m_{y,x}|t$} 
\State \vspace{-4mm}\begin{eqnarray*}
\alpha:\K[x,y]&\rightarrow&\K[x,y]\\
       x&\mapsto& x\\
       y&\mapsto&y-t/m_{y,x}
 \end{eqnarray*}
\Return $\alpha(f)$
\EndIf
\Return $f$\label{End Above}
\end{algorithmic}
\end{spacing}
\end{algorithm}

We now discuss Algorithm \ref{alg:parameter}, which determines the values of the moduli parameters. Let $w=w(T)$ be the weight associated to $\Gamma(T)$. If $\Gamma(T)$ has only one face $\Delta$, then $\supp(f,\Delta)$ is not necessarily equal to $\supp(\NF(T),\Delta)$. We achieve equality by a weighted linear transformation. In the cases  where $\Gamma(T)$ has two faces, equality has already been achieved in Algorithm \ref{alg:clas}. Above $\Gamma(T)$, we then use the method described in the proof of Lemma \ref{lemma:quasiNF} to reduce $f$ modulo $\Jac(f_0)$ where $f_0=\jet(f,\Gamma(T))$: We iteratively apply Algorithm \ref{alg:Transformation} to each term, in the two face case only considering terms in $\Jac(f_0)$, proceeding weighted degree by weighted degree in increasing order (and in each weighted degree according to a total (ordinary) degree ordering).  
After handling a given weighted degree, if Arnold's system for type $T$ contains a monomial $m$ of this degree, we write the sum of the remaining terms in the form
\[\frac{\partial f_0}{\partial x}v_1+\frac{\partial f_0}{\partial y}v_2 +cm,\]
where $v_1,v_2\in\C[x,y]$ are weighted homogeneous, $c\in\C$, and as \[\frac{\partial f_0}{\partial x}v_1+\frac{\partial f_0}{\partial y}v_2,\]
otherwise. By Remark \ref{rmk elimination} this is always possible.
Applying $x\mapsto x-v_1$, $y\mapsto y-v_2$, results in replacing the sum of the remaining terms by a sum of terms which are either in Arnold's system in the $w$-degree under consideration, or of higher $w$-degree. Since $f$ is weighted $d'$-determined, we stop the iteration when we reach degree $d'+1$, where $d'$ is the $w$-degree of the highest $w$-degree monomial in Arnold's system. 

\begin{remark}
In the semi-quasihomogeneous cases, line \ref{call alg 4} in Algorithm \ref{alg:parameter} can be omitted, since the reduction modulo $\Jac(f_0)$ is also handled by lines \ref{line arnold1} to \ref{line arnold end}.
\end{remark}

\begin{remark}
In Algorithm \ref{alg:parameter}, Arnold's system can be replaced by any other choice of a system of the local algebra.
\end{remark}

\begin{remark}
The algebraic extension of $\mathbb{Q}$ introduced for representing the moduli parameters can arise in two steps of the overall algorithm: Reversal of the linear jet in Algorithm \ref{alg:linearTransformation}, and rescaling of the variables at the end of Algorithm \ref{alg:parameter}. Note that the transformation reversing the linear jet is obtained from the factorization  $\jet(f,d)=c g_1^\alpha g_2^\beta g_3^\gamma g_4^\delta$. Here, a field extension can only occur if $\alpha=\beta=2$ and $\gamma=\delta=0$.
\end{remark}

\begin{algorithm}[ht]
\caption{Determine the moduli parameters of a normal form equation of a corank $2$ uni- or bimodal singularities.}%
\label{alg:parameter}
\begin{spacing}{1.05}
\begin{algorithmic}[1]

\Require{$f\in\m^3\subset\K[x,y]$, a germ of modality $1$ or $2$ and corank $2$ of type $T$ over  an algebraic extension field $\K$ of $\Q$, as returned by Algorithm \ref{alg:clas}. In particular, the  set of faces of $\Gamma(T)$ equals the set of faces of $\Gamma(w(T)\dash\jet(f,d(T))$.}

\Ensure{The normal form of $f$, as well as the values of all moduli parameters occuring in a normal form equations that is equivalent to $f$, specified as elements of an algebraic extension field of $\K$.
}
\If{$T=W_{1,\mu-15}^\sharp$ for some $\mu$} \label{line branch deg}
  \Return result of Algorithm \ref{alg:ClassificationDeg} applied to $f$\label{line branch deg2}
\EndIf
\State $w:=w(T)$ and $d:=d(T)$
	\If{$\Gamma(T)$ has exactly one face $\Delta$}
		\State Apply a weighted homogeneous transformation to $f$ such that 				
$\supp(f,\Delta)=\supp(T,\Delta)$.
	\EndIf
	\State $d':=$ highest $w$-degree of a monomial in Arnold's system of $T$
      \State $f_0:=w\dash\jet(f,d)$
	\For{$ j=d+1,\ldots, d'$}
	\For{\textbf{all} terms $t$ of $f$ of $w$-degree $j$  in increasing order by a total degree ordering}\label{above nondeg}
		 \If{$\Gamma(T)$ has exactly one face}
		\State $f:=$ result of Algorithm~\ref{alg:Transformation} with input $f$, $f_0$, $t$ and $(1,1)$, $(1,1)$\label{call alg 4}
		\Else
		 \If{$t\in \Jac (f_0)$}
		    \State $f:=$ result of Algorithm~\ref{alg:Transformation} with input $f$, $f_0$, $t$ and $w_2$, $w_1$
		     \EndIf
		\EndIf
	\EndFor
\If{exists monomial $m$ of $w$-degree $j$ in Arnold's system}\label{line arnold1}
    \State \parbox[t]{\dimexpr\linewidth-\algorithmicindent-\algorithmicindent}{Write $w\dash\jet(f,j)-w\dash\jet(f,j-1)=\frac{\partial f_0}{\partial x}v_1+\frac{\partial f_0}{\partial y}v_2+cm$ with $c\in\K$, $v_1,v_2\in \K[x,y]$ weighted homogeneous.}
\Else
    \State \parbox[t]{\dimexpr\linewidth-\algorithmicindent-\algorithmicindent}{Write $w\dash\jet(f,j)-w\dash\jet(f,j-1)=\frac{\partial f_0}{\partial x}v_1+\frac{\partial f_0}{\partial y}v_2$ with $v_1,v_2\in \K[x,y]$ weighted homogeneous.}\label{line arnold end}
\EndIf
\State  Apply  $x\mapsto x-v_1$, $y\mapsto y-v_2$ to $f$.\label{line arnold4}    
	\EndFor
	\State Delete all terms in $f$ of $w$-degree $>d'$.
\State Apply transformation $x\mapsto ax$, $y\mapsto by$ over an algebraic extension of $\K$ to transform the non-parameter terms to the terms of $\NF(T)$.
\State Read off the parameters $a_i$.
\Return ($\NF(T)$, $(a_i)$)
\end{algorithmic}
\end{spacing}
\end{algorithm}

\section{A Classification Algorithm for Corank $2$, Bimodal Singularities with Degenerate Newton Boundary}\label{section:DegenerateAlg}

In this section we give a classification algorithm for the singularities $W_{1,\mu-15}^\sharp$, where $\mu$ is the Milnor number,  in Arnold's list. They have the property that in all coordinate systems the Newton boundary is degenerate, which is the reason that they have to be treated separately. They are of multiplicity $4$ and the $4$-jet is a $4$-th power of a linear homogeneous polynomial. After a suitable automorphism of $\C[[x,y]]$, we may assume that the corresponding polynomial is of the form
\[
f=(x^2+y^3)^2+\sum\limits_{3i+2j\geq 12+d} w_{ij}x^iy^j\ ,\ d\geq 1.
\]
This automorphism was already constructed in the previous section. Singularities of this type have been studied in \cite{LP}. It is proved that the Milnor number satisfies $\mu(f)\geq 15+d$, and equality holds if and only if $$\sum\limits_{3i+2j=12+d}(-1)^{[i/2]} w_{ij}\neq 0.$$ If the Milnor number $\mu(f)=15+d$ is even, then the germ of the curve defined by $f$ is irreducible with semi--group $\langle 4,6,12+d\rangle$. In the odd case, the curve has two branches. Let $$f=(x^2+y^3)^2+\sum\limits_{3i+2j>12} w_{ij} x^iy^j$$ and assume $\mu:=\mu(f)<\infty$.
Let $>$ be the weighted degree reverse lexicographical ordering with respect to the weights $(3,2)$ on $\C[[x,y]]$ with $x>y$.

In \cite{LP} it is proved that in case of $\mu$ being even the leading ideal of the Jacobian ideal $\langle \frac{\partial f}{\partial x}, \frac{\partial f}{\partial y}\rangle$ is generated by $x^3, x^2y^2, xy^{\frac{\mu-2}{2}}$. If $\mu$ is odd,  then the leading ideal is generated by $x^3, x^2y^2, xy^{\frac{\mu-5}{2}}, y^{\frac{\mu+1}{2}}$.
We obtain a monomial basis of $\C[[x,y]]/\langle \frac{\partial f}{\partial x}, \frac{\partial f}{\partial y}\rangle$ as $\{x^iy^i\}_{(i,j)\in B}$ with 
\[
B=\left\{(i,j) \, \big\vert\,  i\leq 2, j\leq 1\right\}\cup \left\{(i,j) \, \bigg\vert\,  i\leq1, 2\leq j\leq \frac{\mu-4}{2}\right\}
\]
in case that $\mu$ is even and
\[B=\left\{(i,j)\, \big\vert\, i\leq 2, j\leq 1\right\}\cup \left\{(1,j)\, \bigg\vert\, 2\leq j\leq \frac{\mu-7}{2}\right\}\cup \left\{(0,j) \, \bigg\vert\, 2\leq j\leq \frac{\mu-1}{2}\right\},
\]
in case that $\mu$ is odd.
Let $$B_1:=\{(1, \frac{\mu-6}{2}), (1, \frac{\mu-4}{2})\}$$ if $\mu$ is even and $$B_1:=\{0, \frac{\mu-3}{2}), (0, \frac{\mu-1}{2})\}$$ if $\mu$ is odd. 

In \cite{LP}, the following theorem is proved.

\begin{theorem}\label{thm:autom}
There exists an automorphism $\varphi$ of $\C[[x,y]]$ such that $$\varphi(f)=(x^2+y^3)^2+\sum\limits_{(i,j)\in B_1} w_{ij}x^iy^j.$$
\end{theorem}

\begin{note}
In particular, it follows that these singularities are bimodal.
\end{note}

\begin{remark}
The normal form given in this way for the case that the Milnor number is odd differs from Arnold's normal form. Instead of $y^{\frac{\mu-3}{2}}$ and $y^{\frac{\mu-1}{2}}$, he used the monomials $x^2y^{\frac{\mu-9}{2}}$ and $x^2y^{\frac{\mu-7}{2}}$. From a computational point of view, our choice is better. It is easy to convert our normal form to Arnold's normal form. 
See Figure \ref{fig infinite deg}, for an illustration of the normal forms (using our choice of parameter monomials).
\end{remark}

The construction of the automorphism in the theorem is done separately for each weighted degree:
Assume we have already $$f=(x^2+y^3)^2+\sum\limits_{3i+2j\geq 12+a} w_{ij}x^iy^j$$ for some $a$ (with Milnor number $\mu=15+d)$.
If $a<d$, then we have $$\sum\limits_{3i+2j=12+a}(-1)^{[i/2]}w_{ij}=0.$$
This implies that $$\sum\limits_{3i+2j=12+a} w_{ij}x^iy^j=l\cdot (x^2+y^3).$$
We obtain  $$f=\left(x^2+y^3+\frac{1}{2}l\right)^2+\sum\limits_{3i+2j>12+a}\widetilde{w}_{ij}x^iy^j$$ for suitable $\widetilde{w}_{ij}\in \C$. Now we can choose an automorphism $\varphi$ of $\C[[x,y]]$ such that $$\varphi\left(x^2+y^3+\frac{1}{2}l\right)=x^2+y^3 + \text{ terms of weighted degree } \geq \mu$$ (note that we could even find an automorphism mapping $x^2+y^3+\frac{1}{2} l$ to $x^2+y^3)$. We obtain $$\varphi(f)=(x^2+y^3)^2+\sum\limits_{3i+2j>12+a}\overline{w}_{ij}x^iy^j$$
 for suitable $\overline{w}_{ij}\in \C$. 

If $a=d$ then we have $$\sum\limits_{3i+2j=12+d} (-1)^{[i/2]} w_{ij}\neq 0.$$
Similarly as before, we can write
\[
\sum\limits_{3i+2j=12+d} w_{ij}x^iy^j=w_{i_0j_0}x^{i_0}y^{j_0}+l\cdot (x^2+y^3)
\]
with $$(i_0, j_0)=\left\{\begin{array}{ll}(0,\frac{\mu-3}{2}) & \text{if }\mu\text{ is odd}\\
(1, \frac{\mu-6}{2}) & \text{if } \mu \text{ is even}.
\end{array}\right .$$ Since the Milnor number is $15+d$, we obtain $w_{i_0, j_0}\neq 0$. Using a similar automorphism as in the previous case, we may assume with $a_0:=w_{i_0, j_0}$ (the first modulus), that
\[
f=(x^2+y^3)^2+a_0\cdot x^{i_0}y^{j_0}+\sum\limits_{3i+2j>12+d} w_{ij}x^iy^j\ .
\]
Note, that $12+d=\mu-3$, and we have to compute the normal form of $f$ up to degree $\mu-1$.
Now we can write 
$$
\sum\limits_{3i+2j=13+d} w_{ij}x^iy^j=e\cdot x^{i_1}y^{j_1}+l\cdot (x^2+y^3)$$ with $$(i_1, j_1)=\left\{\begin{array}{ll}(1,\frac{\mu-5}{2}) & \text{if }\mu\text{ is odd}\\
(0, \frac{\mu-2}{2}) & \text{if } \mu \text{ is even}.
\end{array}\right .$$
Using an automorphism as before, we may assume that $l=0$. 

If $e=0$, we are done with weighted degree $\mu-2$.

If $e\neq 0$ we define an automorphism $\varphi$ of $\C[[x,y]]$ by the exponential of the vector field $$\delta=c\cdot (3y^2\frac{\partial}{\partial x}-2x\frac{\partial}{\partial y})$$ with $$c=(-1)^{\mu-1}\frac{e}{(\mu-3)a_0}.$$
Since by construction, $\varphi(x^2+y^3)=x^2+y^3$, we obtain $$\varphi(f)=(x^2+y^3)^2+a_0 \cdot x^{i_0}y^{j_0}+\sum\limits_{3i+2j\geq 14+d}\widetilde{w}_{ij}x^iy^j$$ for suitable $\widetilde{w}_{ij}\in\C$.

\begin{remark}
Note, that for practical purposes, we have to compute $\varphi$ only up to weighted degree $5$, and apply it to $a_0 \cdot x^{i_0}y^{j_0}+\sum\limits_{3i+2j=13+d} w_{ij}x^iy^j$ since we know that $\varphi((x^2+y^3)^2)=(x^2+y^3)^2$.
\end{remark}

Now let $$(i_1, j_1)=\left\{\begin{array}{ll}(0, \frac{\mu-1}{2}) & \text{if } \mu \text{ is odd}\\
(1,\frac{\mu-4}{2}) & \text{if }\mu\text{ is even}\end{array}\right .$$ and write $$\sum\limits_{3i+2j=14+d}\widetilde{w}_{ij} x^iy^j=a_1 \cdot x^{i_1} y^{j_1}+l\cdot (x^2+y^3).$$ Using an automorphism as in the first case, we may assume $l=0$, and obtain as normal form 
\[
(x^2+y^3)^2+a_0 x^{i_0} y^{j_0}+a_1 \cdot x^{i_1} y^{j_1}.
\]
We summarize the approach in Algorithm \ref{alg:ClassificationDeg}.

\begin{algorithm}[h]
\caption{Algorithm to determine parameters for the bimodal singularities of type $W_{1,\mu-15}^\sharp$.}%
\label{alg:ClassificationDeg}
\begin{spacing}{1.05}
\begin{algorithmic}[1]

\Require{$f=\gamma\cdot (\alpha x^2+\beta y^3)^2+$ terms of weighted $(3,2)$-degree $>12\, \in \K[x,y]$ with $\alpha, \beta, \gamma \in \K$  and $\mu:=\mu(f)<\infty$.}
\Ensure{$A$ normal form of $f$ of the form $$\begin{array}{lcl}(x^2+y^3)^2+a_0\cdot xy^{\frac{\mu-6}{2}}+a_1 \cdot xy^{\frac{\mu-4}{2}} & \text{if} & \mu \text{ is even}\\
(x^2+y^3)^2+a_0\cdot y^{\frac{\mu-3}{2}}+a_1\cdot y^{\frac{\mu-1}{2}}& \text{if} & \mu \text{ is odd}\end{array}$$ with $a_0\neq 0$, as well as the corresponding moduli parameters of a normal form equation defined over an algebraic extension field of 
$\K$.}

  \State Apply transformation $x\mapsto ax$, $y\mapsto by$ over an algebraic extension field of $\K$ to $f$ to transform the weighted homogeneous part of $f$ to  $(x^2+y^3)^2$.
\State Let > be the local weighted degree reverse lexicographical ordering with weights $(3,2)$ and $x>y$.
\State Compute a standard basis $G$ of $\langle\frac{\partial f}{\partial x}, \frac{\partial f}{\partial y}\rangle$ with respect to $>$.
\State Compute $\mu$ the Milnor number of $f$, and set $d:=\mu-15$.
\State $a:=13$
\While{$a<12+d$}
\State $g:=$ weighted homogeneous part of $f$ of degree $a$
\State Write $g=l\cdot(x^2+y^3)$.
\State Construct automorphism $\varphi$ with $\varphi(x^2+y^3+\frac{1}{2}l)=x^2+y^3$ up to degree $\mu-1$.
\State $f:=\varphi(f)\ ,\ a:=a+1$
\EndWhile
\State $g:=$ weighted homogeneous part of $f$ of degree $12+d$
\If{$\mu$ is odd}
\State $m_0:=y^{\frac{\mu-3}{2}}\ ,\ m_1: = xy^{\frac{\mu-5}{2}}\ ,\ m_2: = y^{\frac{\mu-1}{2}}$
\Else
\State $\ m_0:=xy^{\frac{\mu-6}{2}}\ ,\ m_1: = y^{\frac{\mu-2}{2}}\ ,\ m_2: = xy^{\frac{\mu-4}{2}}$
\EndIf
\State Write $g=a_0\cdot m_0+l\cdot (x^2+y^3)$.
\State Construct automorphism $\varphi$ with $\varphi(x^2+y^3+\frac{1}{2} l)=x^2+y^3$ up to degree $\mu-1$.
\State $f:=\varphi(f)$
\State $g:=\text{ weighted homogeneous part of } \varphi \text{ of degree } 13+d$
\State Write $g=e\cdot m_1+l\cdot (x^2+y^3)$.
\State Construct automorphism $\varphi$ with $\varphi(x^2+y^3+\frac{1}{2}l)=x^2+y^3$ up to degree $\mu-1$.
\State $f:=\varphi(f)$
\If{$e\neq 0$}
\State $c:=(-1)^{\mu-1}\frac{e}{(\mu-3)a}$
\State Construct automorphism $\varphi$ defined by the vector field $c\cdot (3y^2\frac{\partial}{\partial x}- 2x\frac{\partial}{\partial y})$ up to degree $5$.
\State $f:=(x^2+y^3)^2+\varphi(f-(x^2+y^3)^2)$
\State $g:=\text{ the weighted homogeneous part of } f \text{ of degree } 14+d$
\State Write $g=a_1 \cdot m_2+l\cdot (x^2+y^3)$.
\EndIf
\Return ($\NF(W_{1,\mu-15}^\sharp$), $(a_0,a_1)$
)
\end{algorithmic}
\end{spacing}
\end{algorithm}

\begin{remark}
The approach described in Algorithm~\ref{alg:parameter} in case of a non-degenerate Newton boundary can be adapted to also handle the cases $W_{1,\mu-15}^\sharp$.  
However, this strategy requires more iterations than Algorithm~\ref{alg:ClassificationDeg}.
To adapt Algorithm~\ref{alg:parameter}, we remove lines~\ref{line branch deg} and~\ref{line branch deg2}, and in line \ref{call alg 4}, we call Algorithm~\ref{alg:parameterDeg} instead of Algorithm \ref{alg:Transformation} if $f$ is of type $W_{1,\mu-15}^\sharp$. 

Note that in these cases Algorithm \ref{alg:clas} does not require a field extension, hence, Algorithm~\ref{alg:parameterDeg} is called with input defined over $\Q$. Note also that Algorithm~\ref{alg:parameterDeg} is applicable with any choice of a system $B$ of the local algebra.

\end{remark}

\begin{algorithm}[ht]
\caption{Remove terms above the diagonal in cases with degenerate Newton boundary.}%
\label{alg:parameterDeg}
\begin{spacing}{1.05}
\begin{algorithmic}[1]

\Require{$f,f_0\in\Q [x,y]$, $t\in\Q [x,y]$ a term, and weights $u_1,u_2\in\mathbb Z^2$.}
\Ensure{$h\in\Q [x,y]$ such that $f\sim h$.}

\State $w:=w(f)$ and $j:=w\dash\deg(t)$
\State $g:=$ output of Algorithm \ref{alg:Transformation} with input $f$, $f_0$, $t$ and $u_1, u_2$
\State $B:=$ Arnold's  system of $\Q[x,y]/\Jac(f)$
\If{$t\in\Jac(f_0)$ or $g\neq f$ or ($g=f$ and $B$ contains an element of degree $j$) 
}
    \Return g
\EndIf
	\State $m:=$ monomial in $B$ of minimal $w$-degree
\State Factorize $f_0=\gamma \cdot g_0^2$ over $\Q$ with $\gamma \in \Q$ and $g_0\in \Q [x,y]$ linear.
	\State $\phi:=$ automorphism defined by $(\frac{\partial g_0}{\partial y}\frac{\partial}{\partial x}-\frac{\partial g_0}{\partial x}\frac{\partial}{\partial y})$ up to $w\dash$degree $5$
	\State $s:= \coeff(f,m)\cdot m$
\State $t':=w\dash\jet(\phi(s)-s,j)$
\For{\textbf{all} terms $\tilde t$ of $t'$ in increasing order by standard degree}
\State $t':=-f_0 +$ result of Algorithm~\ref{alg:Transformation} with input $t'+f_0$, $f_0$, $\tilde t$ and $u_1$, $u_2$\label{call alg 4a}
\State  $t':=w\dash\jet(t',j)$
\EndFor
\State $c:=-t/t'$
\State $\phi_c:=$ automorphism defined by $c\cdot (\frac{\partial g_0}{\partial y}\frac{\partial}{\partial x}-\frac{\partial g_0}{\partial x}\frac{\partial}{\partial y})$ up to $w\dash$degree $5$.
\State $h:=f_0 + \phi_c(f-f_0)$
\For{\textbf{all} terms $\tilde t$ of $h$ of $w$-degree $j$ in increasing order by standard degree}
\State $h:=$ result of Algorithm~\ref{alg:Transformation} with input $h$, $f_0$, $\tilde t$ and $u_1$, $u_2$\label{call alg 4b}
\EndFor
\Return $h$
\end{algorithmic}
\end{spacing}
\end{algorithm}

\begin{figure}[h]
\begin{center}
\begin{tabular}
[c]{ccc}%
{\includegraphics[
width=1.5in
]%
{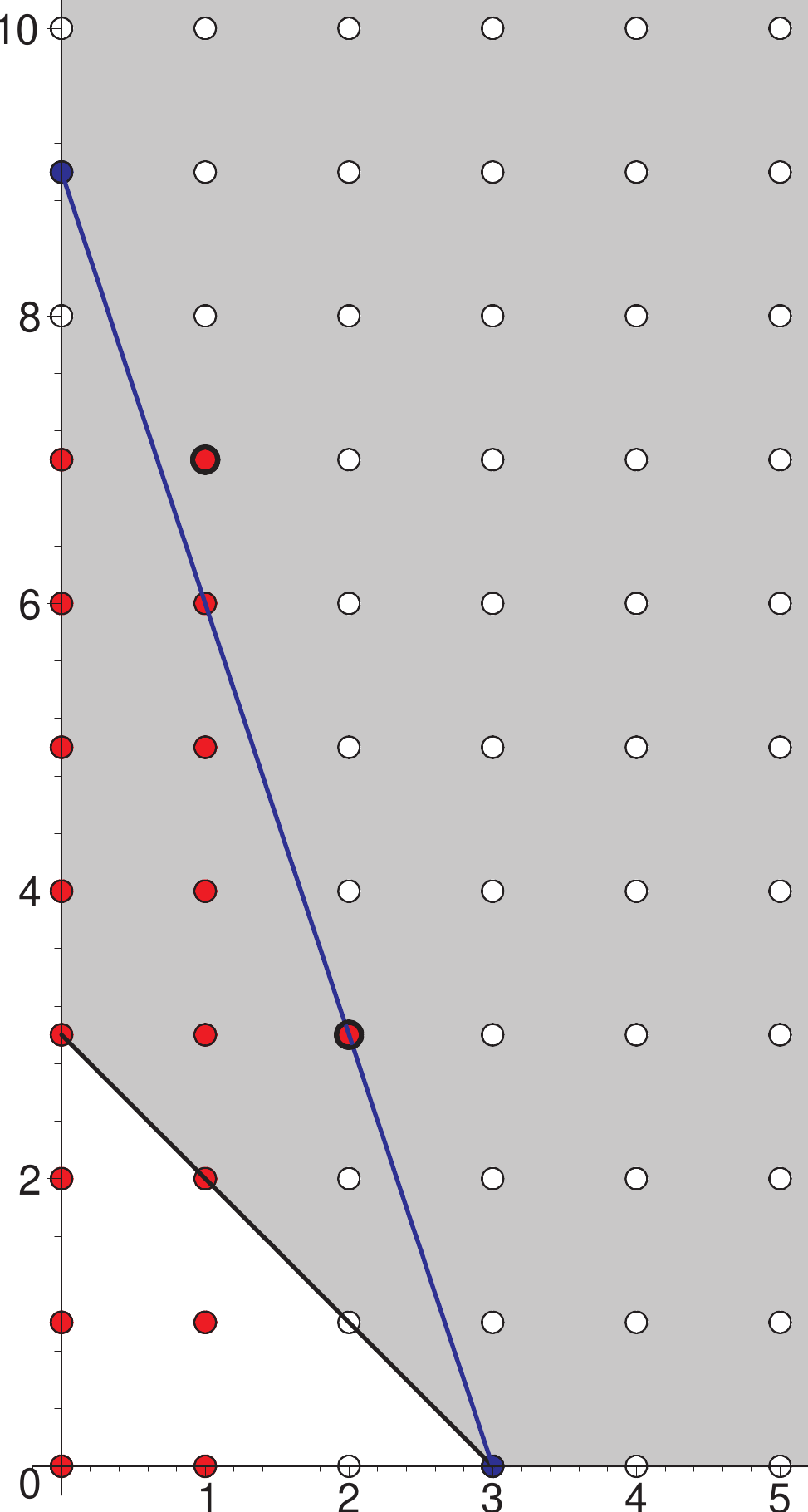}%
}%
&
{\includegraphics[
width=1.5in
]%
{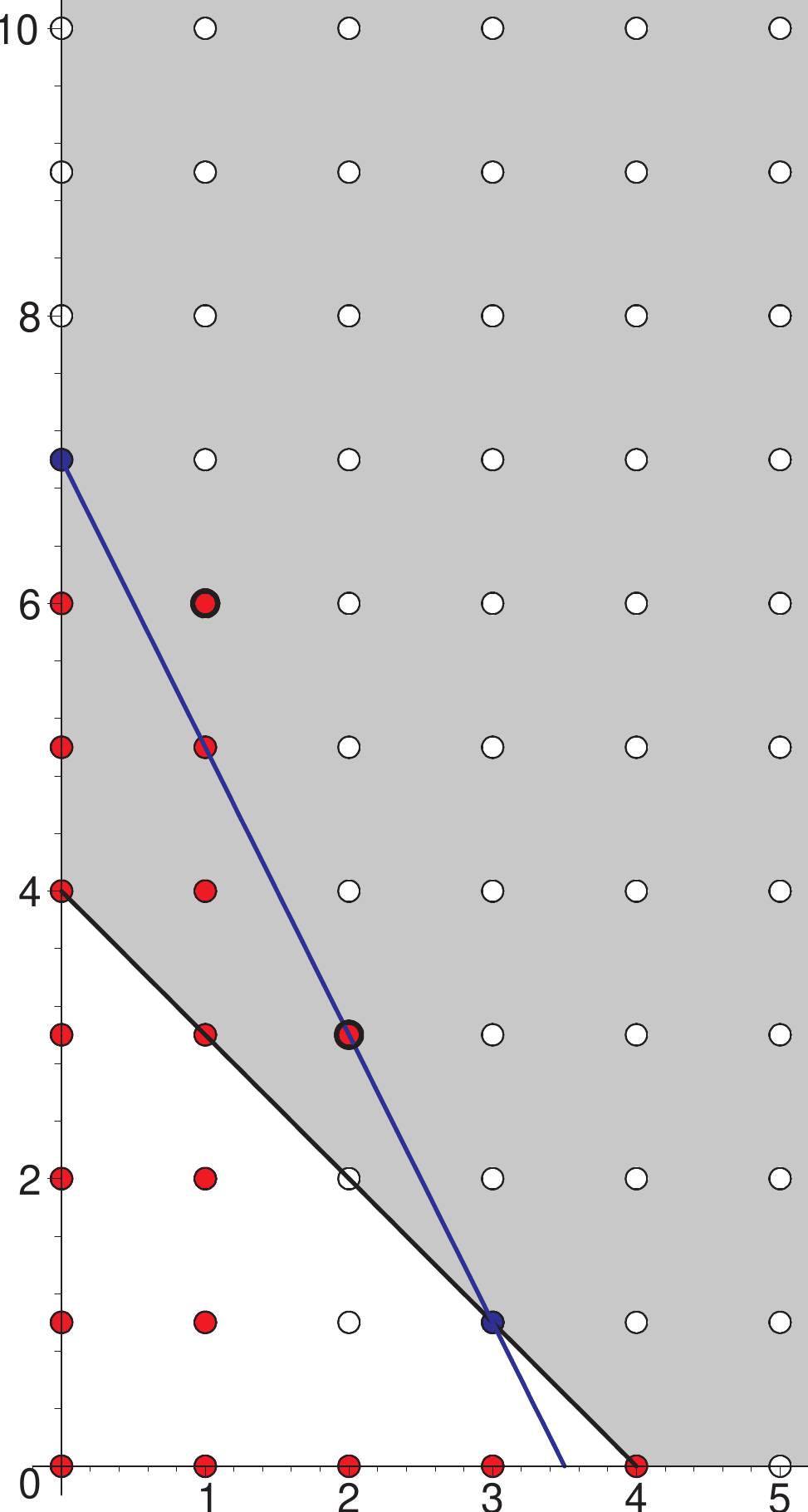}%
}%
&
{\includegraphics[
width=1.5in
]%
{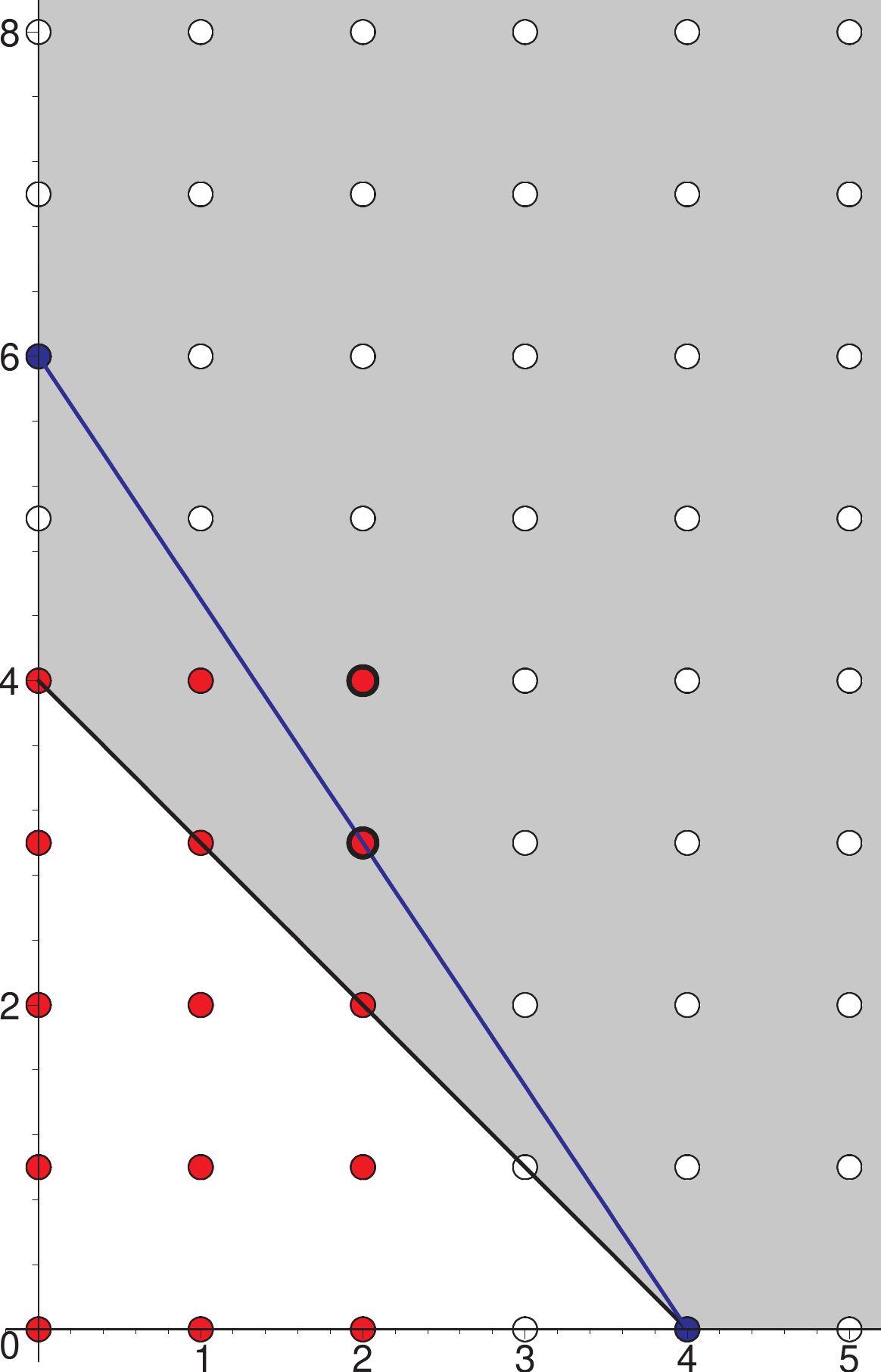}%
}%
\\
$J_{3,0}$ & $Z_{1,0}$ & $W_{1,0}$\\
&  & \\%
{\includegraphics[
width=1.5in
]%
{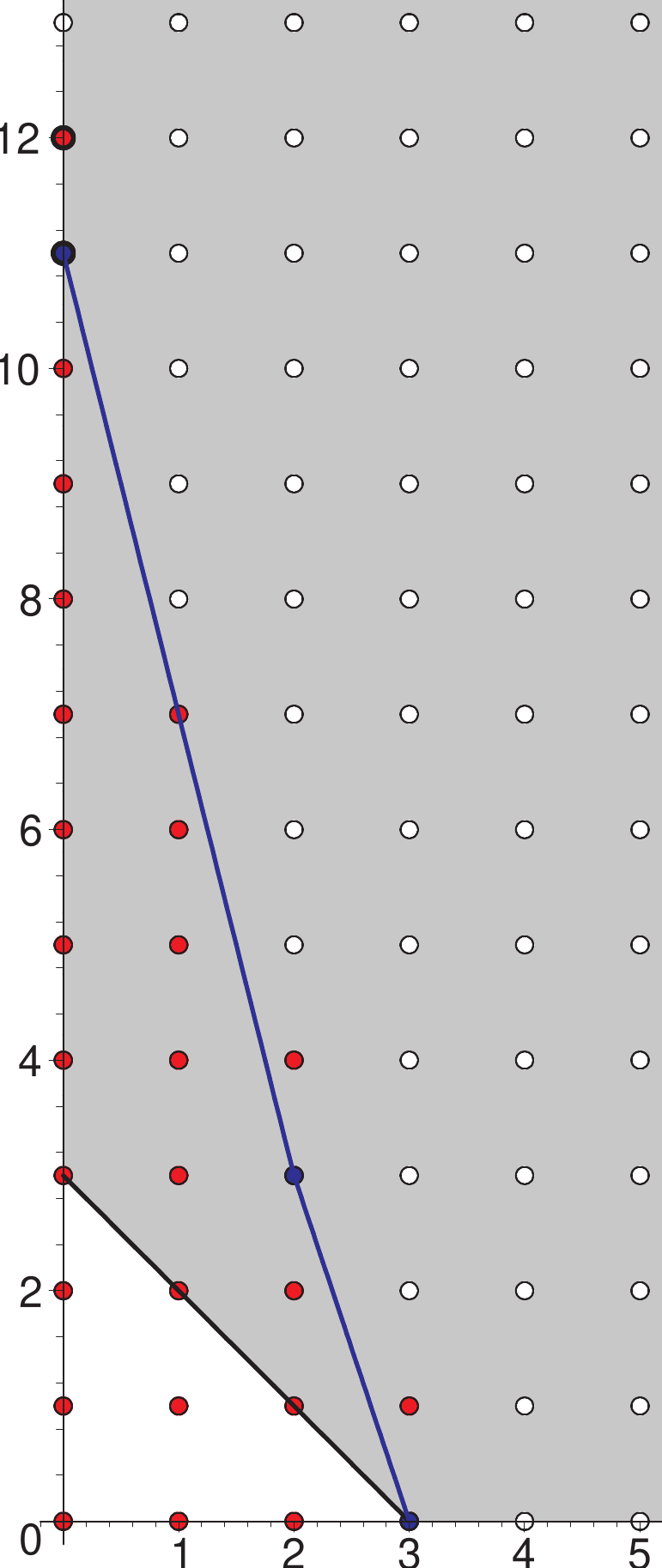}%
}%
&
{\includegraphics[
width=1.5in
]%
{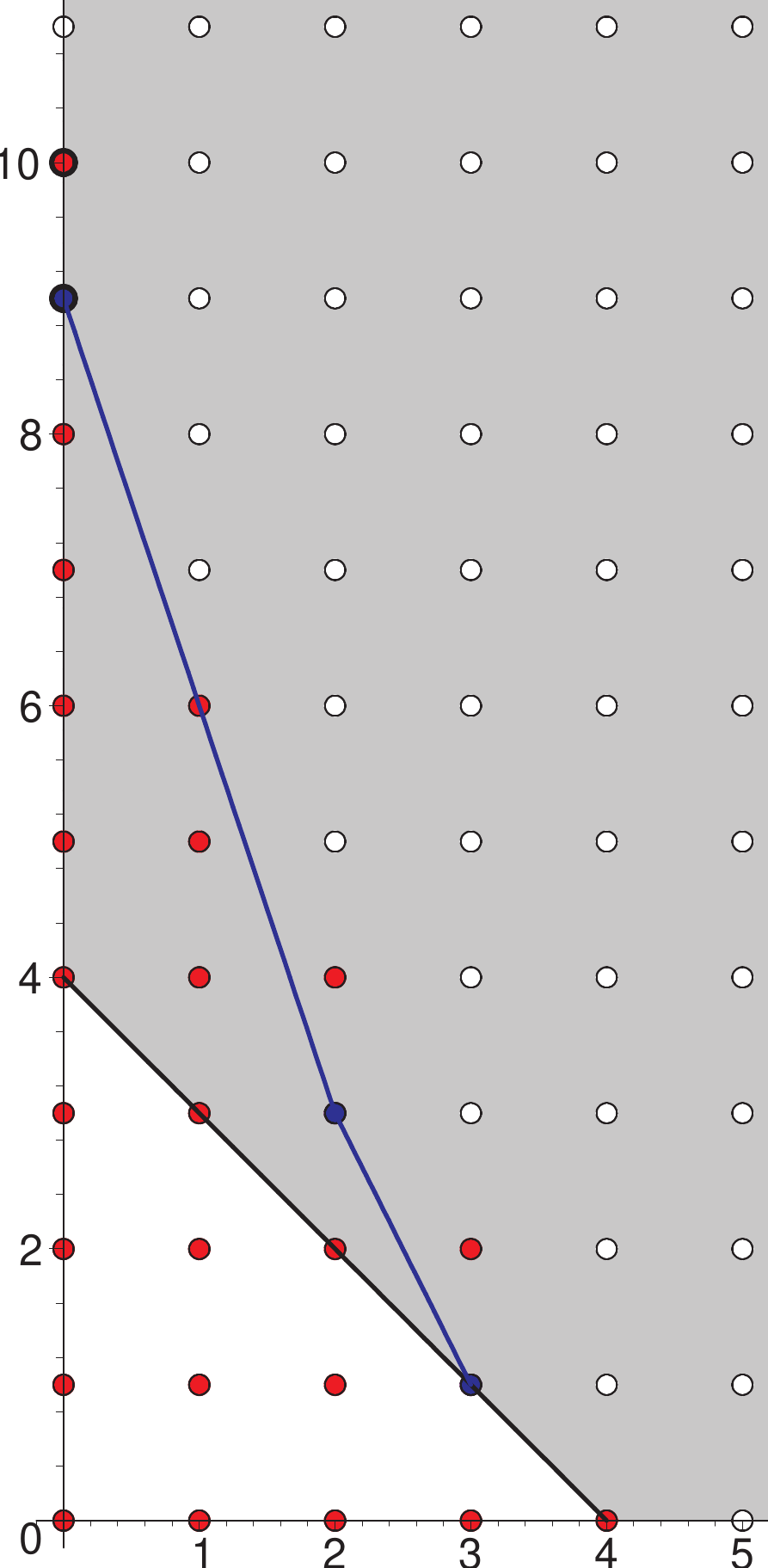}%
}%
&
{\includegraphics[
width=1.5in
]%
{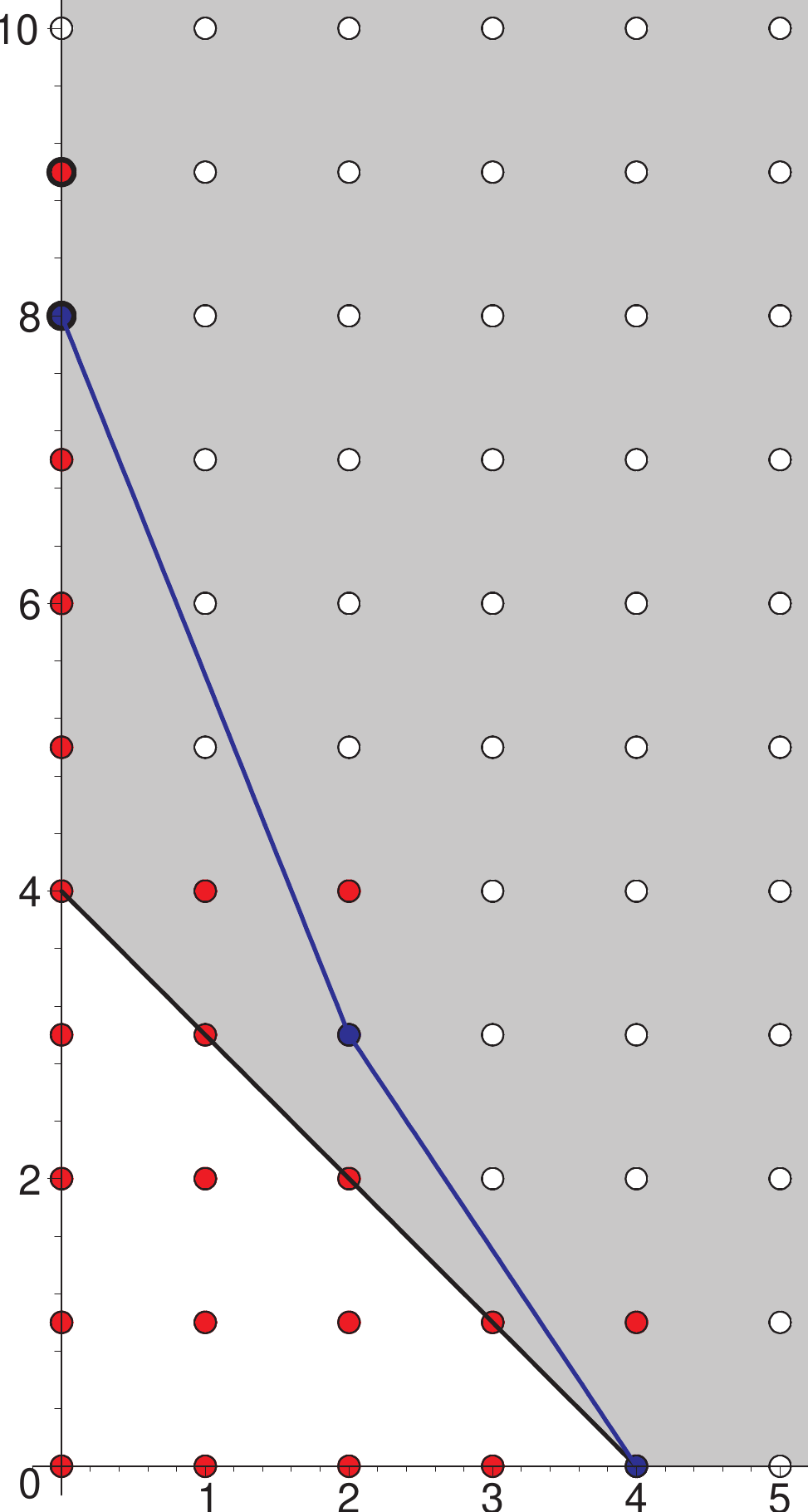}%
}%
\\
$J_{3,p}$ for $p=2$ & $Z_{1,p}$ for $p=2$ & $W_{1,p}$ for $p=2$%
\end{tabular}
\end{center}
\caption{Infinite series of bimodal corank $2$ singularities with non-degenerate Newton boundary.}%
\label{fig infinite}
\end{figure}

\begin{figure}
\begin{center}
\begin{tabular}
[c]{ccc}%
{\includegraphics[
width=1.5in
]%
{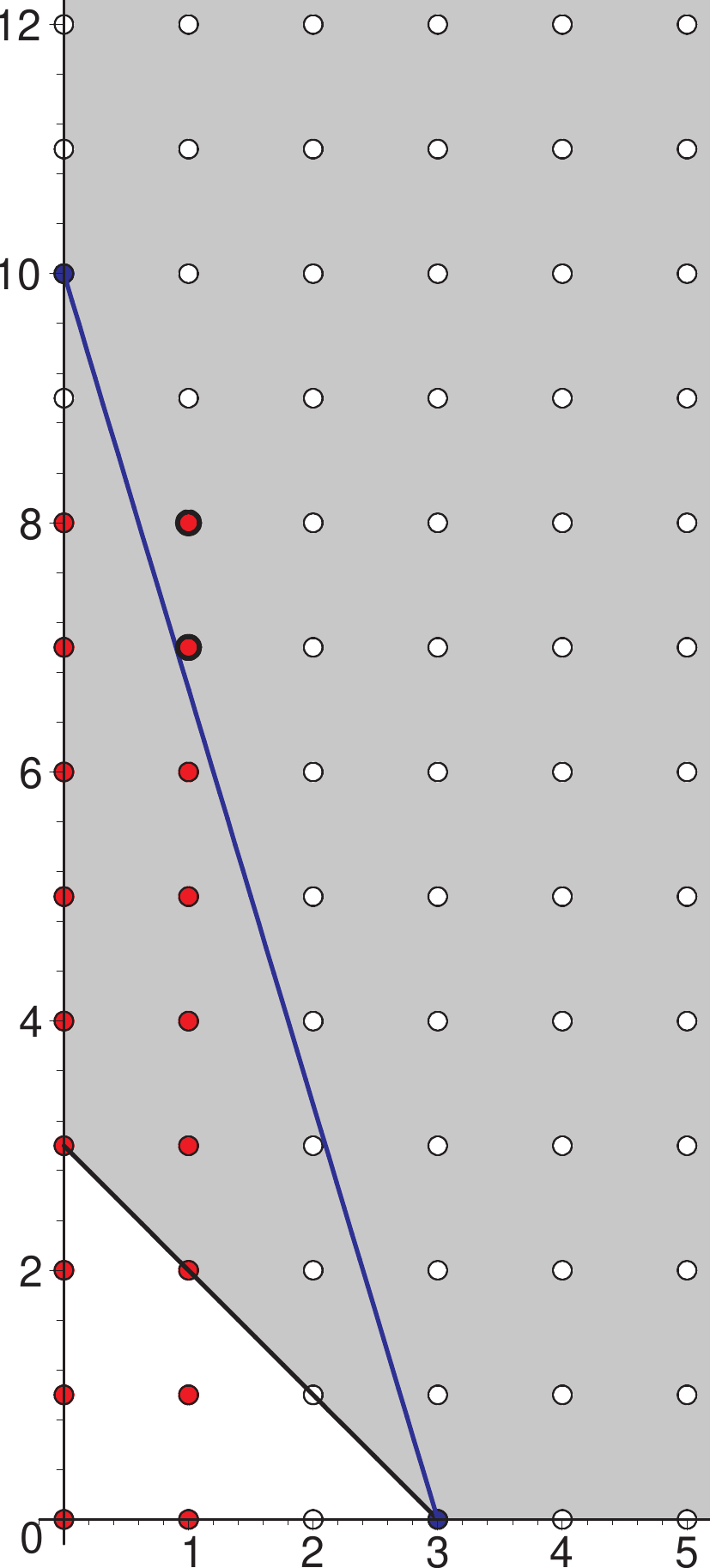}%
}%
&
{\includegraphics[
width=1.5in
]%
{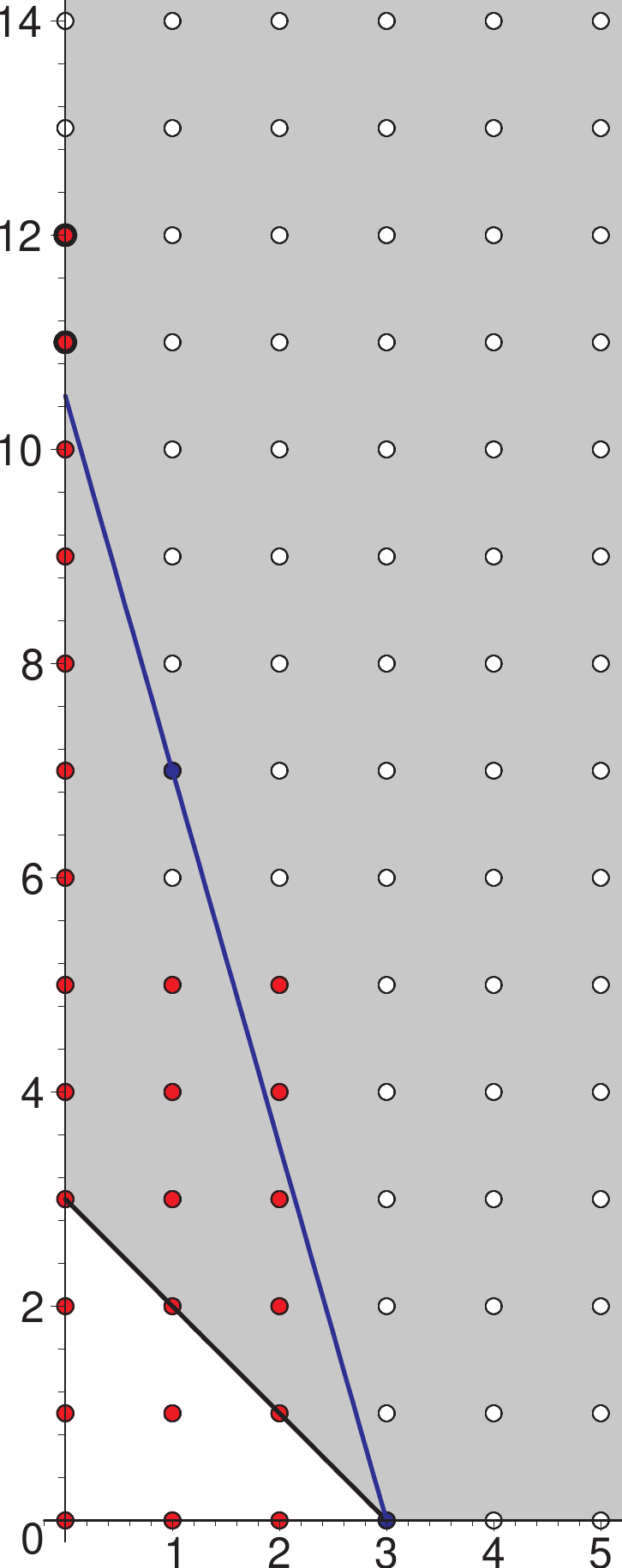}%
}%
&
{\includegraphics[
width=1.5in
]%
{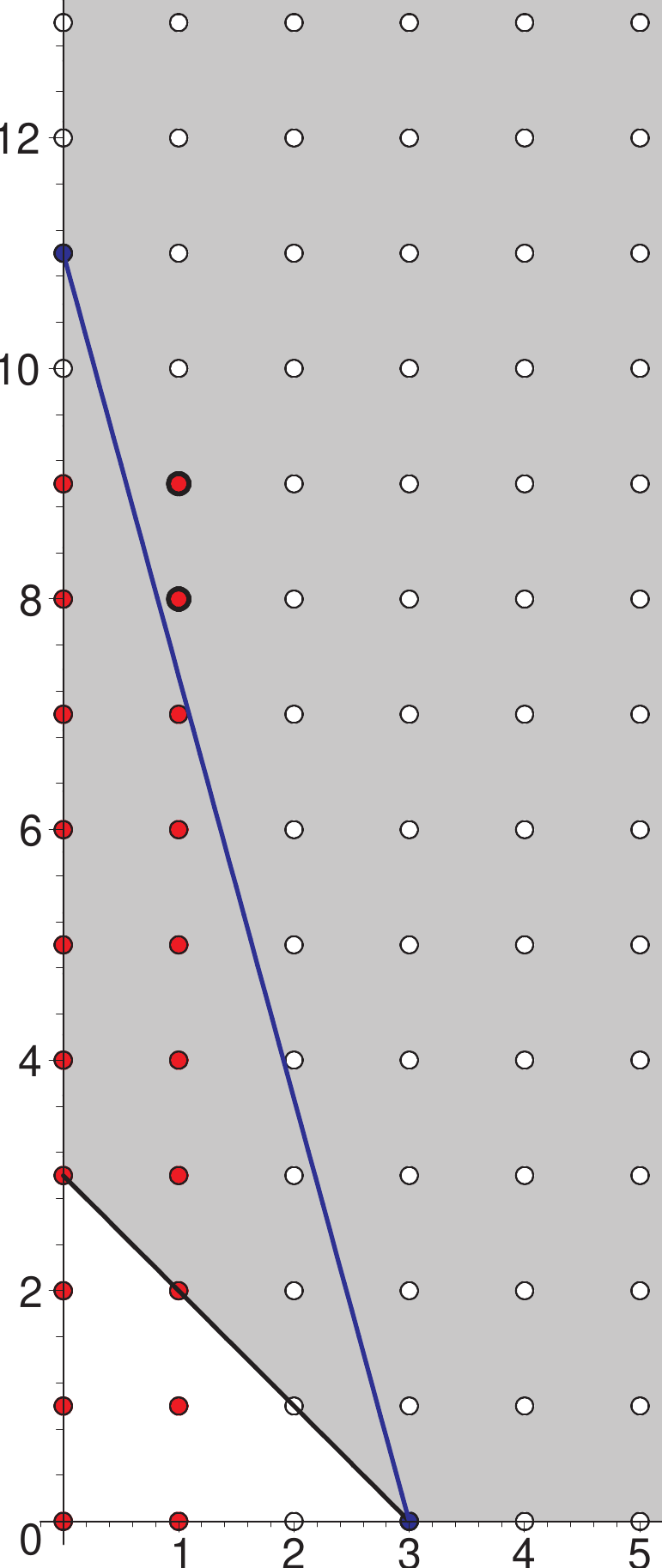}%
}%
\\
$E_{18}$ & $E_{19}$ & $E_{20}$%
\end{tabular}
\end{center}
\caption{Exceptional bimodal corank $2$ singularities of type E.}%
\label{fig E}
\end{figure}

\begin{figure}
\begin{center}
\begin{tabular}
[c]{ccc}%
{\includegraphics[
width=1.5in
]%
{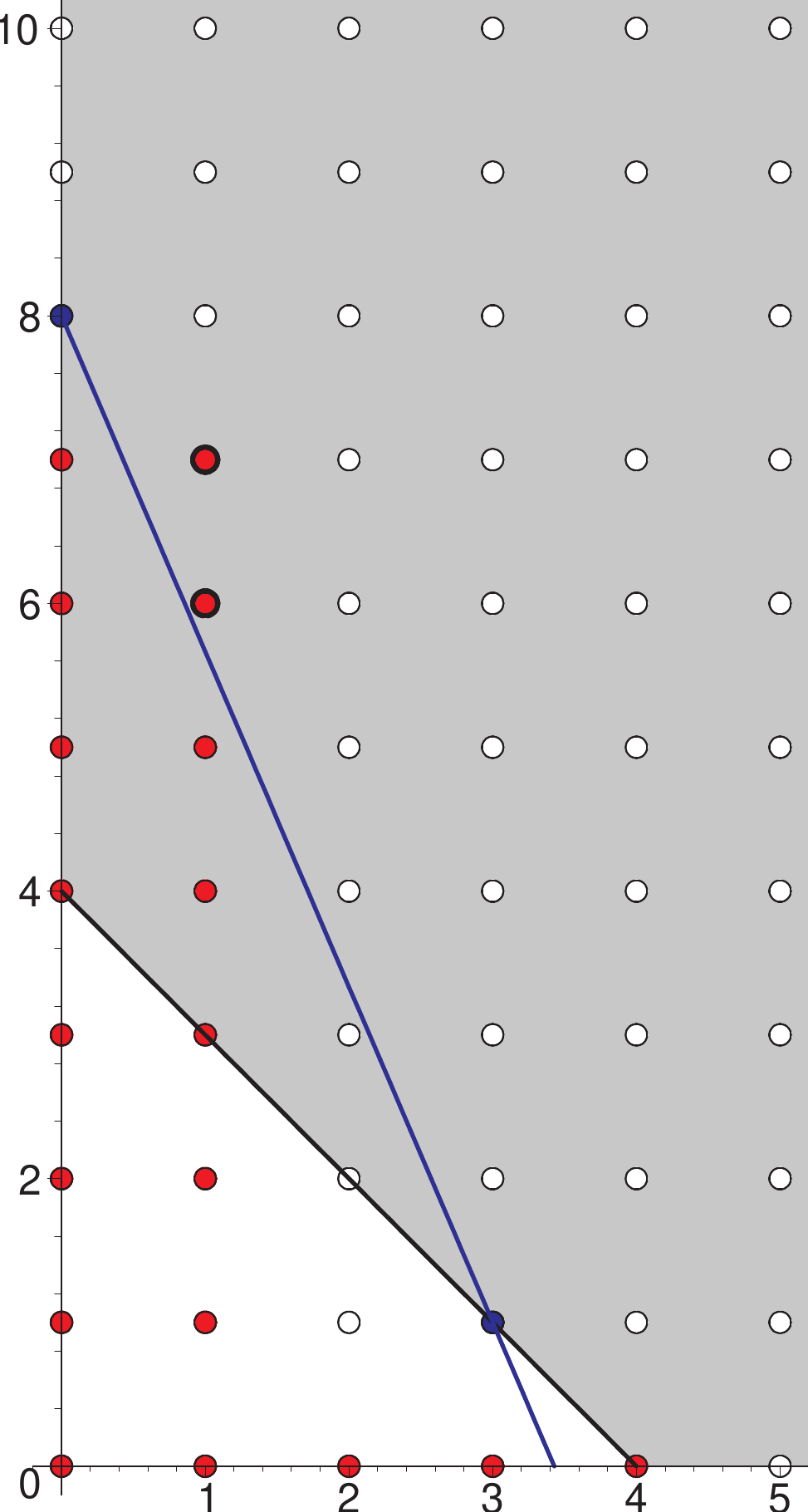}%
}
&
{\includegraphics[
width=1.5in
]%
{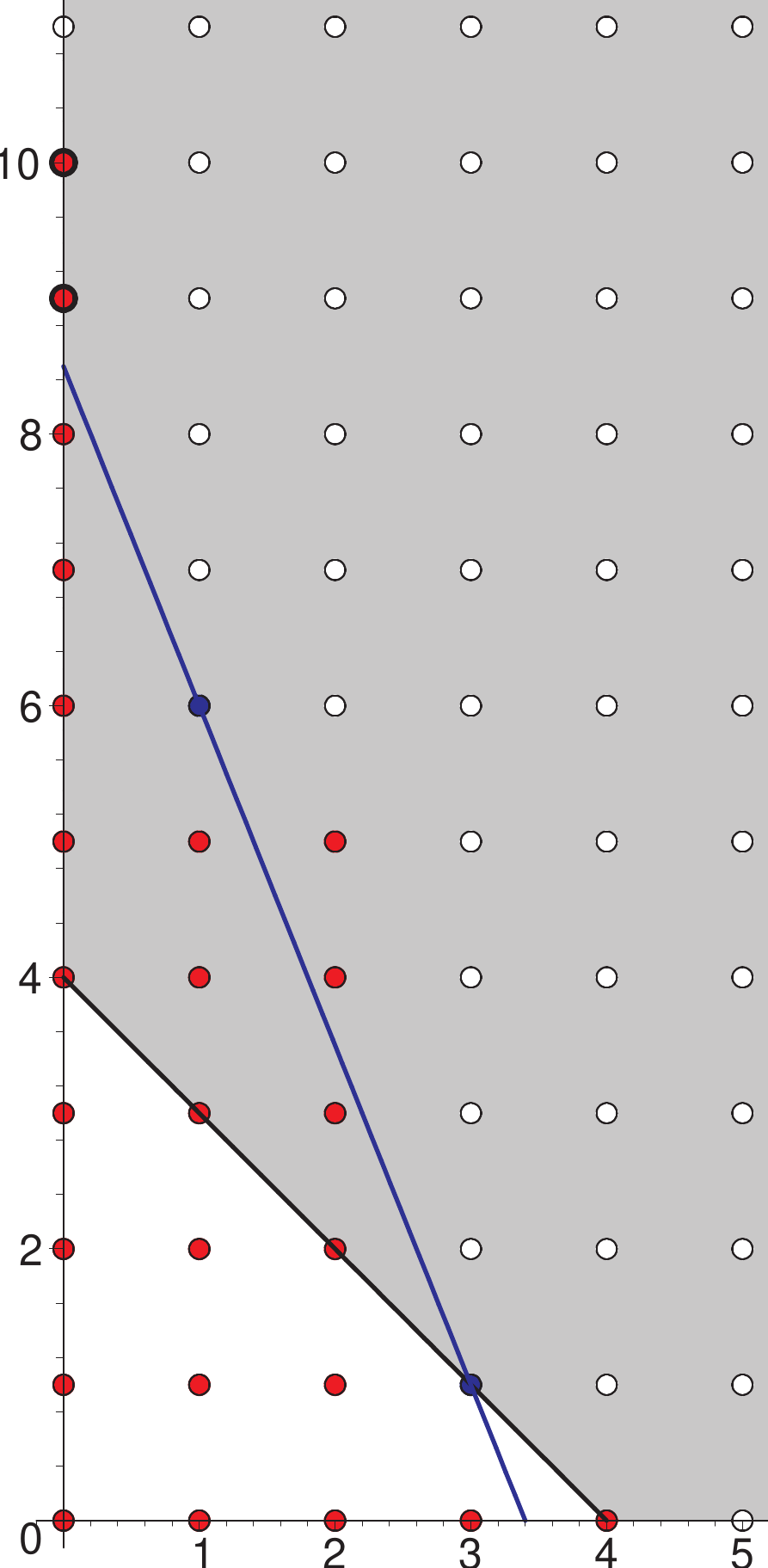}%
}
&
{\includegraphics[
width=1.5in
]%
{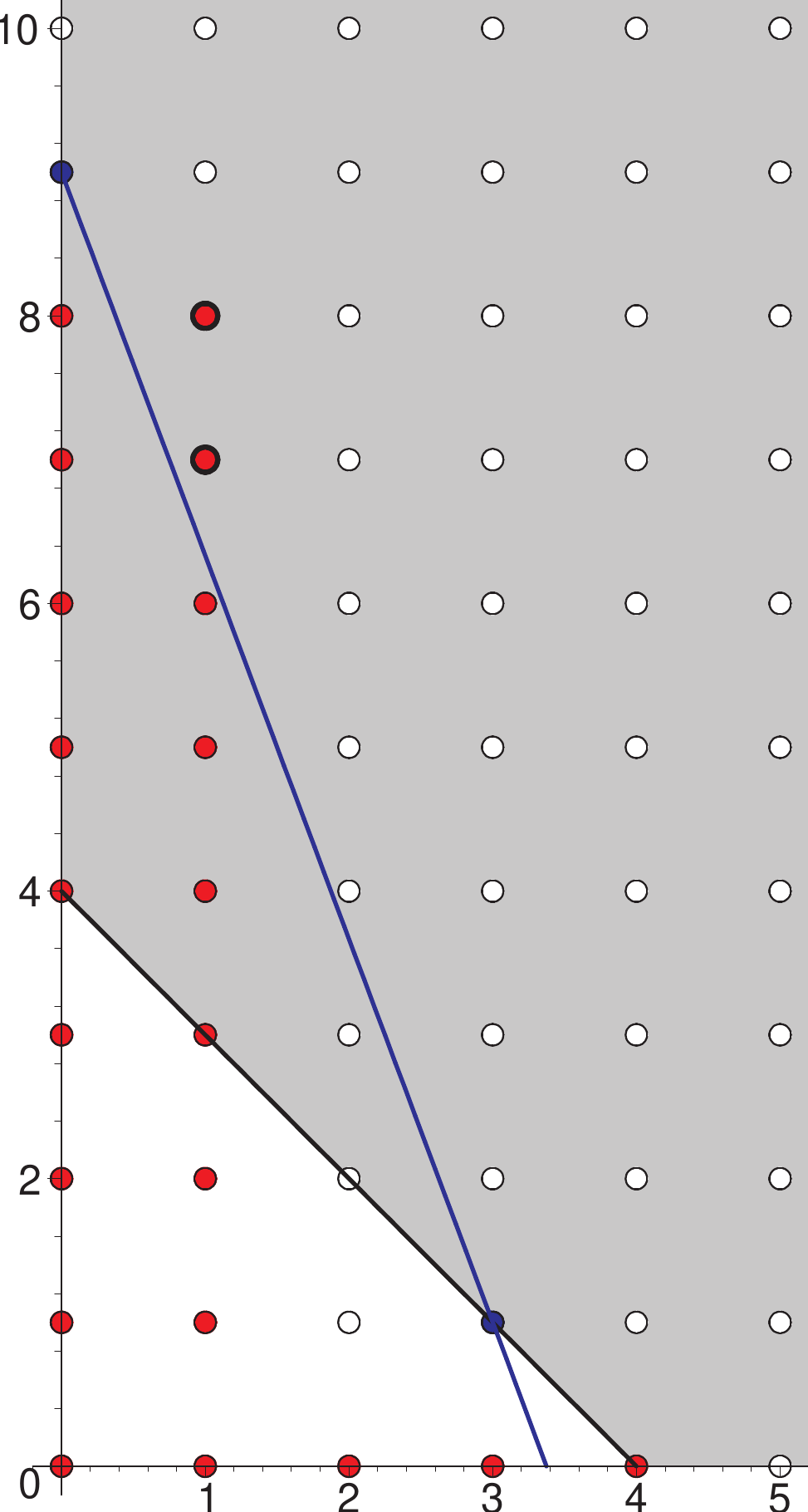}%
}
\\
$Z_{17}$ & $Z_{18}$ & $Z_{19}$%
\end{tabular}
\end{center}
\caption{Exceptional bimodal corank $2$ singularities of type Z.}%
\label{fig Z}
\end{figure}

\begin{figure}
\begin{center}
\begin{tabular}
[c]{cc}%
{\includegraphics[
width=1.5in
]%
{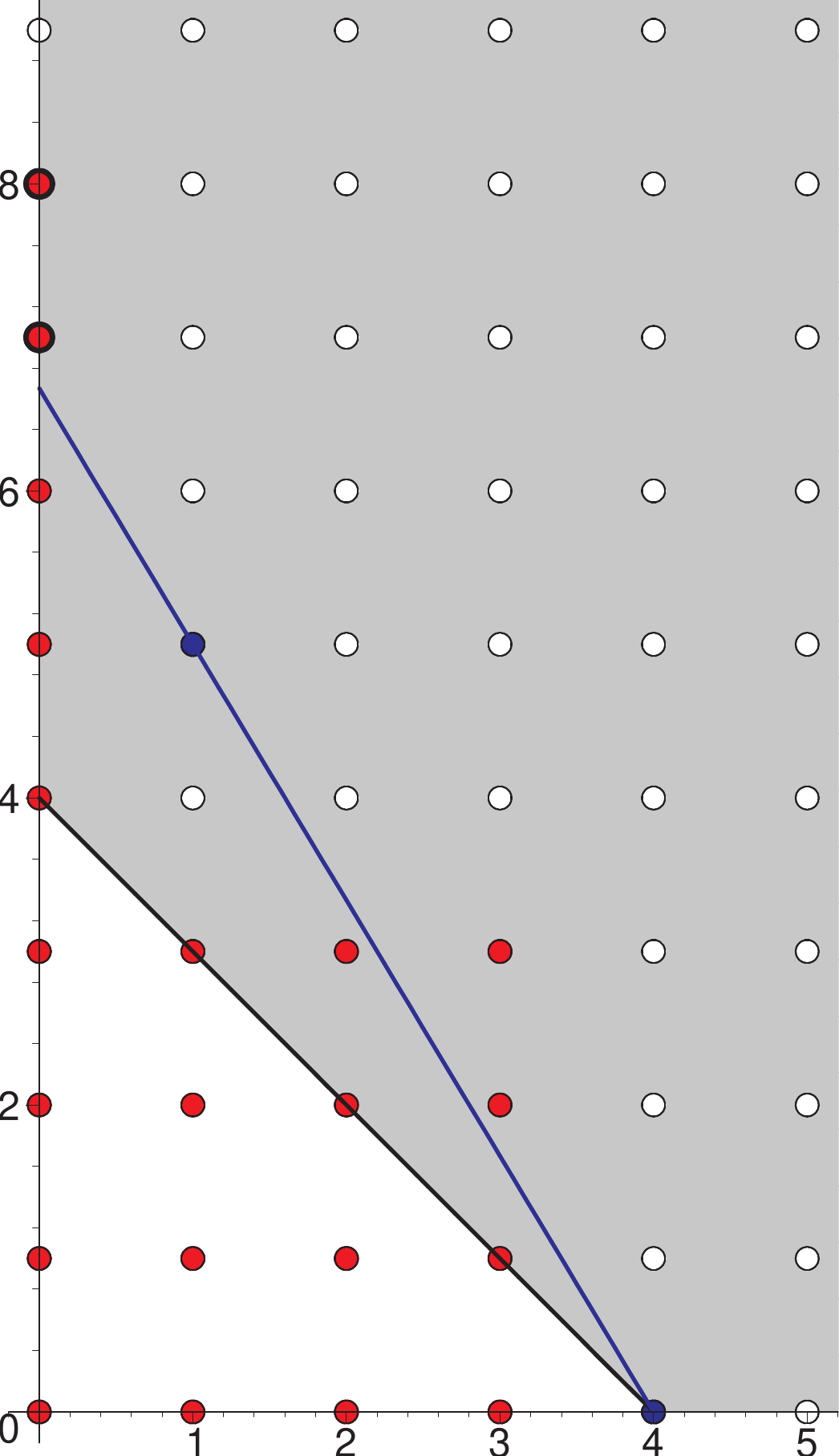}%
}
&
{\includegraphics[
width=1.67in
]%
{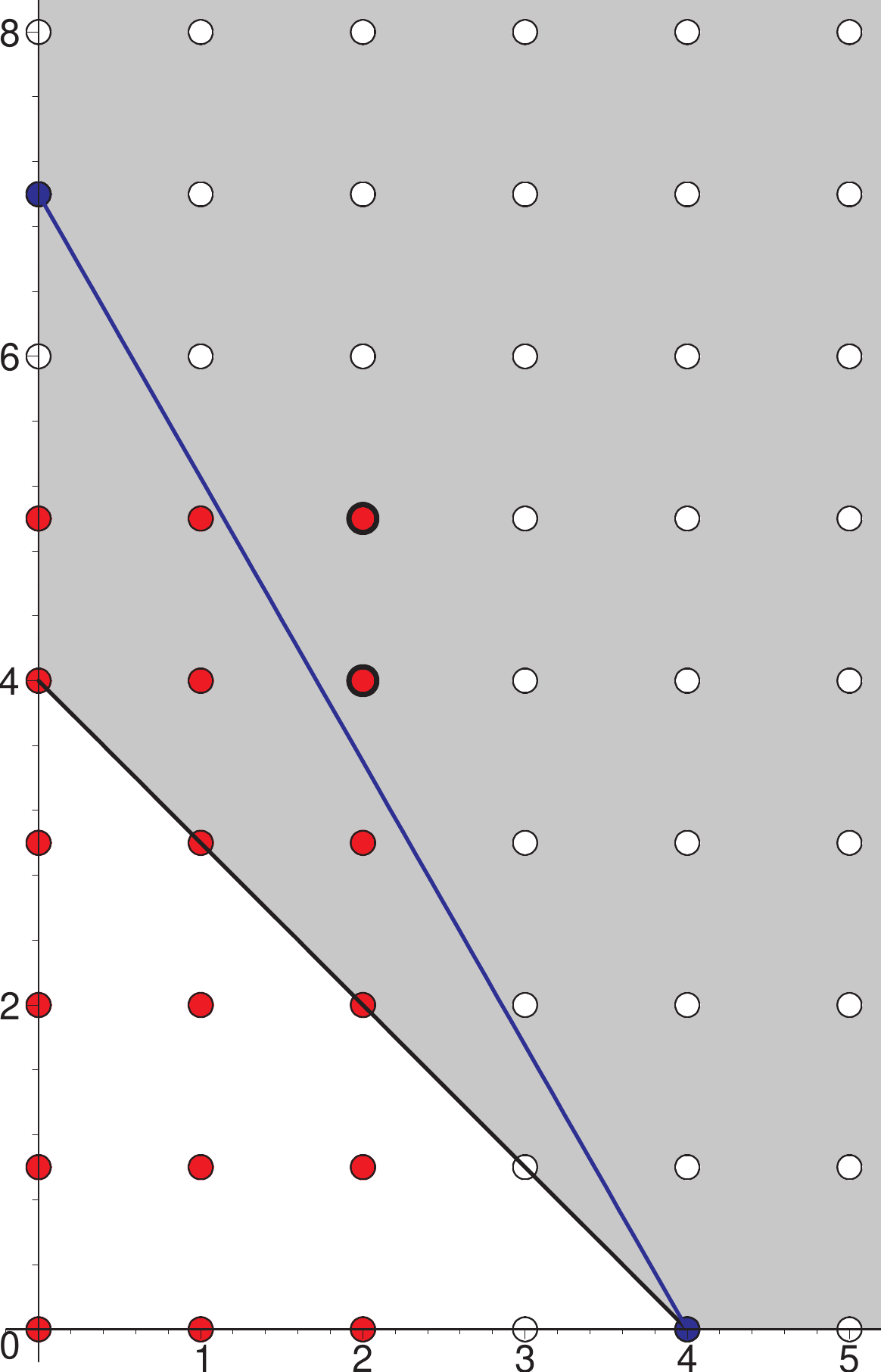}%
}
\\
$W_{17}$ & $W_{18}$%
\end{tabular}
\end{center}
\caption{Exceptional bimodal corank $2$ singularities of type W.}%
\label{fig W}
\end{figure}

\begin{figure}
\begin{center}
\begin{tabular}
[c]{cc}%
{\includegraphics[
width=2.11in
]%
{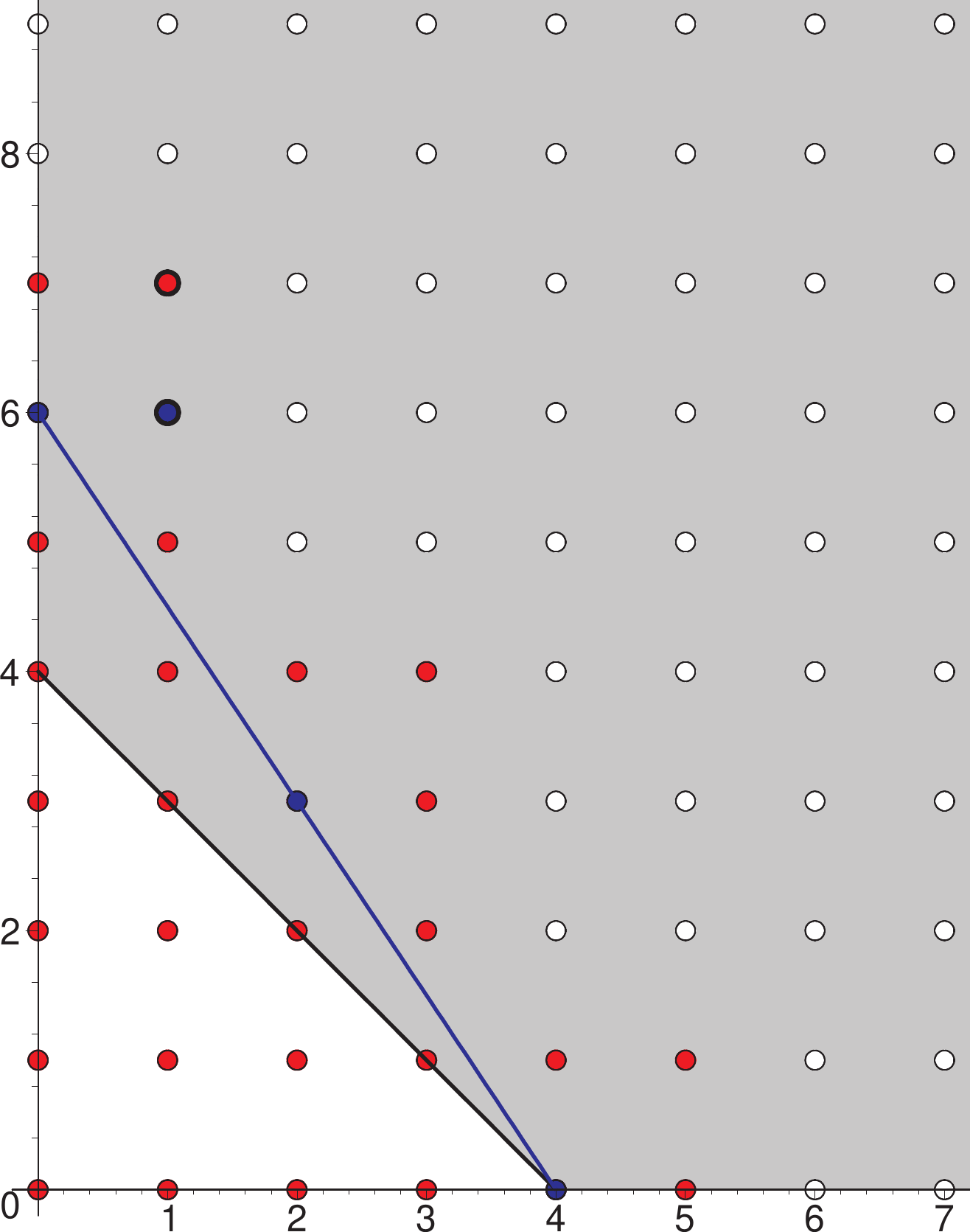}%
}
&
{\includegraphics[
width=2.4105in
]%
{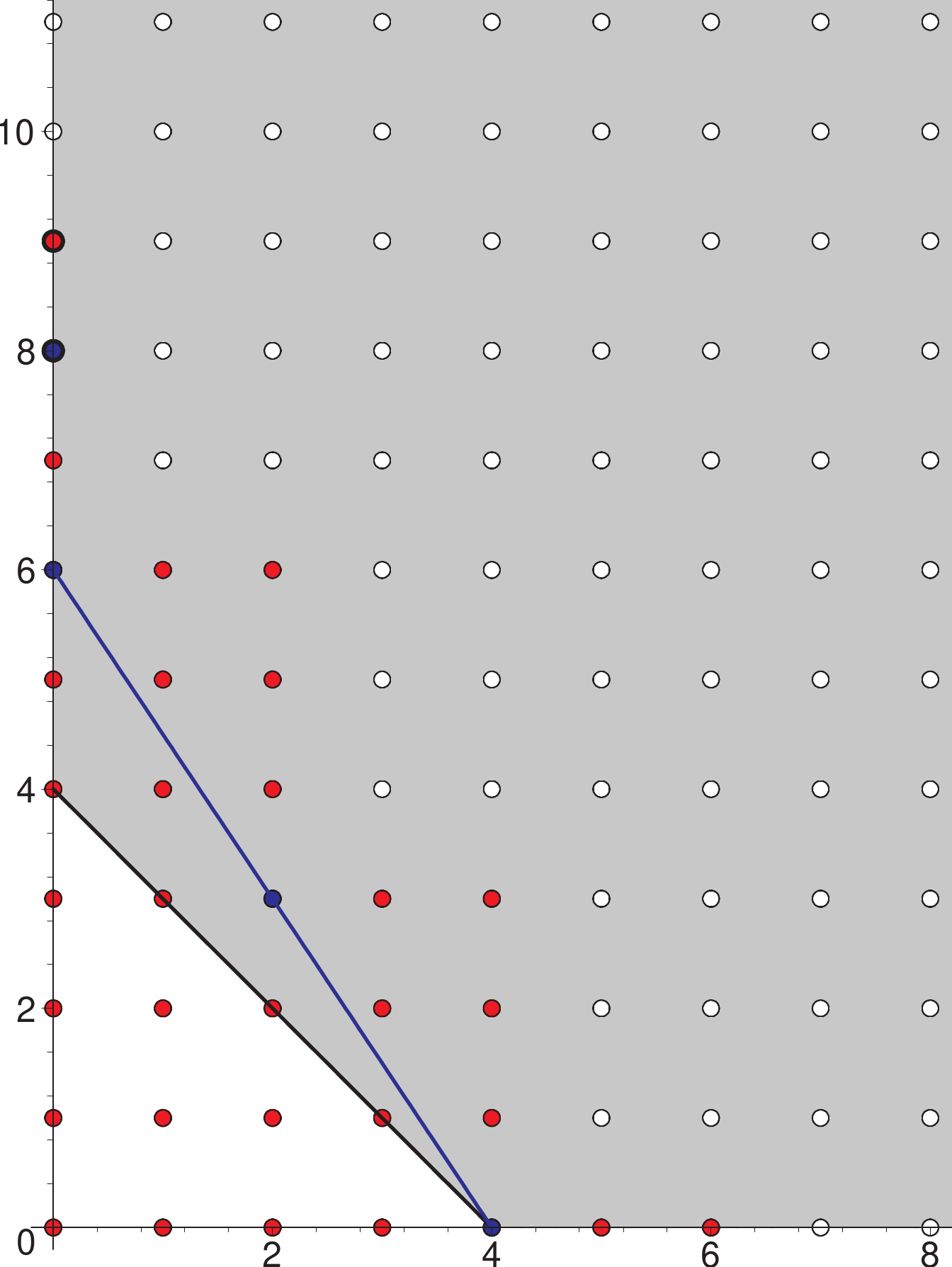}%
}
\\
$W_{1,2q-1}^{\sharp}$ for $q=2$ & $W_{1,2q}^{\sharp}$  for $q=2$%

\end{tabular}
\end{center}
\caption{Infinite series of bimodal corank $2$ singularities with degenerate Newton boundary.}%
\label{fig infinite deg}
\end{figure}

\end{document}